\documentclass[11pt,reqno,oneside]{amsart}
\usepackage{graphicx, subfigure,}
\usepackage{amssymb}
\usepackage{epstopdf}
\usepackage{amssymb}
\usepackage[a4paper, total={6in, 9.1in}]{geometry}

\usepackage{bm}
\usepackage{amsmath,amsfonts,amsthm,mathrsfs,amssymb,cite}
\usepackage[usenames]{color}

\usepackage{graphics}
\usepackage{framed}

\usepackage{enumerate}
\usepackage{hyperref}
\usepackage{xcolor}
\hypersetup{
  colorlinks,
  linkcolor={blue!80!black},
  urlcolor={blue!80!black},
  citecolor={blue!80!black}
}
\usepackage[capitalize,noabbrev]{cleveref}

\numberwithin{equation}{section}

\newtheorem{theorem}{Theorem}[section]
\newtheorem{lemma}[theorem]{Lemma}
\newtheorem{proposition}[theorem]{Proposition}

\theoremstyle{definition}
\newtheorem{defn}{Definition}[section]

\theoremstyle{remark}
\newtheorem{remark}[theorem]{Remark}

\numberwithin{equation}{section}

\newcounter{saveeqn}

\newcommand{\ba}{\begin{array}}
\newcommand{\ea}{\end{array}}

\newcommand{\bea}{\begin{eqnarray*}}
\newcommand{\eea}{\end{eqnarray*}}
\newcommand{\bean}{\begin{eqnarray}}
\newcommand{\eean}{\end{eqnarray}}

\def\Oh{\mathcal O}

\title[Determining anomalies in a semilinear equation]{Determining anomalies in a semilinear elliptic equation by a minimal number of measurements}

\author{Huaian Diao}
\address{School of Mathematics, Jilin University, Changchun 130012, China}
\email{diao@jlu.edu.cn}

\author{Xiaoxu Fei}
\address{School of Mathematics, Central South University, Changsha, China}
\email{feixx0921@163.com}

\author{Hongyu Liu}
\address{Department of Mathematics, City University of Hong Kong, Kowloon, Hong Kong SAR, China}
\email{hongyu.liuip@gmail.com, hongyliu@cityu.edu.hk}

\author{Li Wang}
\address{Department of Mathematics, City University of Hong Kong, Kowloon, Hong Kong SAR, China}
\email{1322771004@qq.com}

\date{} 
\begin{document}

\begin{abstract}

We are concerned with the inverse boundary problem of determining anomalies associated with a semilinear elliptic equation of the form $-\Delta u+a(\mathbf x, u)=0$, where $a(\mathbf x, u)$ is a general nonlinear term that belongs to a H\"older class. It is assumed that the inhomogeneity of $f(\mathbf x, u)$ is contained in a bounded domain $D$ in the sense that outside $D$, $a(\mathbf x, u)=\lambda u$ with $\lambda\in\mathbb{C}$. We establish novel unique identifiability results in several general scenarios of practical interest. These include determining the support of the inclusion (i.e. $D$) independent of its content (i.e. $a(\mathbf{x}, u)$ in $D$) by a single boundary measurement; and determining both $D$ and $a(\mathbf{x}, u)|_D$ by $M$ boundary measurements, where $M\in\mathbb{N}$ signifies the number of unknown coefficients in $a(\mathbf x, u)$. The mathematical argument is based on microlocally characterising the singularities in the solution $u$ induced by the geometric singularities of $D$, and does not rely on any linearisation technique.

\medskip



\end{abstract}

\maketitle

\section{Introduction}

\subsection{Mathematical setup and summary of major findings}

Initially focusing on the mathematics, but not the physics, we introduce the forward boundary value problem associated with a semilinear elliptic equation:
\begin{equation}\label{eq:helm1}
\Delta u+a(\mathbf x, u)=0\quad\mbox{in}\ \ \Omega, \quad u|_{\partial\Omega}=\psi,
\end{equation}
where
\begin{enumerate}
\item $\Omega$ is a bounded Lipschitz domain in $\mathbb{R}^n$, $n=2, 3$, and $\psi\in H^{1/2}(\partial\Omega)$;

\item there is a bounded Lipschitz domain $D\Subset\Omega$ such that $\Omega\backslash\overline{D}$ is connected, and a constant $\lambda\in\mathbb{C} $ such that
\begin{equation}\label{eq:aaa1}
a(\mathbf x, u)=(f(\mathbf x, u)-\lambda u)\chi_D+\lambda u,\quad x\in \Omega. 
\end{equation}
That is, $a(\mathbf{x}, u)=f(\mathbf{x}, u)$ in $D$, whereas $a(\mathbf x, u)=\lambda u$ in $\Omega\backslash\overline{D}$. Furthermore, we suppose that $
	 a(\mathbf x, u)$ is $ C^1$-continuous with respect to $u$ for a fixed $\mathbf x\in \Omega$ and $\partial_u a(\mathbf x, u) \in L^\infty (\Omega ).$

\item $f(\mathbf x,z):\ (\mathbf x, z)\in D \times\mathbb{C}\mapsto\mathbb{C}$ fulfils the following admissibility conditions:
\begin{enumerate}

\item For $u(\cdot)\in H^1(\Omega)$, $f(\mathbf{x}, u(\cdot))\in L^2(\Omega)$;

\item $f(\mathbf x, z)$ is ${C}^\gamma$-continuous, $\gamma\in (0, 1)$, with respect to $(\mathbf x, z)\in D\times\mathbb{C}$;

\item $f(\mathbf x, z)$ fulfills that for a proper $\psi\in H^{1/2}(\partial \Omega )$, there exists a solution $u\in H^1(\Omega)$ to \eqref{eq:helm1}.
\end{enumerate}
In such a case, we say that $f$ belongs to the admissible class $\mathscr{A}$ and write $f\in\mathscr{A}$ or $(D; f)\in\mathscr{A}$ to signify the support of the inhomogeneity of $f$ is $D$.
\end{enumerate}

In what follows, we assume that $\lambda$ is known, which characterises the homogeneous space $\Omega\backslash\overline{D}$, whereas $(D; f)$ is unknown, which is referred to as an anomalous inhomogeneous inclusion. In this paper, we aim to study the following inverse boundary problem:
\begin{equation}\label{eq:ip1}
\Lambda_{D, f}(\psi):=(\psi|_{\partial\Omega}, \partial_\nu u|_{\Omega}),\ \psi\in H^{1/2}(\partial\Omega)\ \mbox{fixed}\longrightarrow D\quad \mbox{independent of $f$},
\end{equation}
where $u\in H^1(\Omega)$ is a solution to \eqref{eq:helm1}, and $\nu\in\mathbb{S}^{n-1}:=\{\mathbf x\in\mathbb{R}^n; |\mathbf x|=1\}$ is the exterior unit normal vector to $\partial\Omega$. In the physical context, $D$ signifies the support of the anomalous inhomogeneity whereas $f$ characterises its physical content. Hence, the inverse problem \eqref{eq:ip1} is concerned with recovering the location and shape of the anomalous inhomogeneity independent of its content. It is also referred to as the inverse inclusion problem in the theory of inverse problems. Furthermore, we also study the following inverse boundary problem:
\begin{equation}\label{eq:ip2}
\Lambda_{D, f}(\psi_j):=(\psi_j|_{\partial\Omega}, \partial_\nu u_j|_{\Omega}),\ \psi_j\in H^{1/2}(\partial\Omega),\ j=1,\ldots, N\in\mathbb{N} \longrightarrow \mbox{both $D$ and $f$},
\end{equation}
where $u_j\in H^1(\Omega)$ is a solution to \eqref{eq:helm1} associated with the boundary data $u_j|_{\partial\Omega}=\psi_m$. Here, $N\in\mathbb{N}$ signifies the number of unknown coefficients of $f(x, u)$, say e.g. $f(x, u)=\sum_{j=1}^N \lambda_j u^j$ with $\lambda_j\in\mathbb{C}$. That is, for the inverse problem \eqref{eq:ip2}, we aim at recovering both the support and its physical content of the inhomogeneous inclusion by $N$ boundary measurements. It can be verified that both inverse problems \eqref{eq:ip1} and \eqref{eq:ip2} are formally determined; that is, the cardinalities of the unknown inclusion and the known boundary data are equal. By cardinality, we mean the number of independent variables in a quantity. Hence, we refer to them as inverse problems with a minimal number of measurements, or simply minimal boundary measurements.

It is emphasised that we only assume the existence of a solution to \eqref{eq:helm1} and do not assume the uniqueness of the solution. That is, there might exist multiple solutions to \eqref{eq:helm1}. Associated with a single $\psi\in H^{1/2}(\partial\Omega)$, $\Lambda_\psi$ is referred to as a single pair of Cauchy data, or a single boundary measurement. Throughout, we always assume that $\psi$ is properly chosen such that
\eqref{eq:helm1} has a solution $u\in H^1(\Omega)$. By the admissibility of $f$, one can easily infer from the standard interior regularity estimate for elliptic PDEs that $u\in H^2(\Omega')$ for any $\Omega'\Subset\Omega$ (cf. \cite{McL}).

For the inverse inclusion problem \eqref{eq:ip1}, we mainly consider its unique identifiability issue. That is, we aim at establishing the sufficient conditions under which $D$ can be uniquely determined by $\Lambda_{D, f}(\psi)$ in the sense that if two admissible inclusions $(D_m; f_m)$, $m=1,2$, produce the same boundary measurement, i.e. $\Lambda_{D_1, f_1}(\psi)=\Lambda_{D_2, f_2}(\psi)$ associated with a fixed $\psi\in H^{1/2}(\partial\Omega)$, then one has $D_1= D_2$. The main results that we establish in this paper for the inverse problem \eqref{eq:ip1} can be roughly summarised as follows:
\begin{enumerate}
\item Under a generic condition, a local unique identifiability result is established showing that the difference of the supports of two nonlinear anomalies cannot possess corner or conic singularities;

\item If certain a-prior information is available on $D$, say e.g. it is a convex polygon or polyhedron or of a corona-shape, it can be uniquely determined.

\item In several practical scenarios, say e.g. nonlinear anomalies are embedded in linear anomalies in a layered manner or certain multi-layered/nest nonlinear anomalies, we show that under generic conditions, one can determine the support of each layer by a single measurement within convex polygonal/polyhedral geometries.
\end{enumerate}

Similarly, for the inverse problem \eqref{eq:ip2}, we establish unique identifiability results in three scenarios:
\begin{enumerate}
\item If $f(x, u)=\sum_{j=1}^N \lambda_j u^j$ with $\lambda_j\in\mathbb{C}$ and $D$ is of polygonal/polyhedral or corona-shape, then under generic conditions, we can establish the unique identifiability result in determining both $D$ and $f$ by using $N$ measurements.

\item If the anomalous inclusion is of a layered/nest structure with $f$ in each layer of the form given in (1) above (distinct among different layers), we can establish the unique identifiability result in determining both $D$ and $f$ by using minimal boundary measurements. 

\end{enumerate}

\subsection{Physical motivation and background discussion}

In the physical context, the PDE system \eqref{eq:helm1} can be used to describe several physical problems of practical importance, especially in the wave scattering theory. For example, if one takes
\begin{equation}\label{eq:phy1}
\lambda=k^2\quad \mbox{with}\ \ k\in\mathbb{R}_+;\quad f(\mathbf x, u)=k^2 q_1(\mathbf x) u\quad \mbox{with}\ \ q_1\in L^\infty(D),
\end{equation}
 \eqref{eq:helm1} is the classical Helmholtz system, which describes the transverse time-harmonic electromagnetic scattering when $n=2$ \cite{LL1}, and the time-harmonic acoustic scattering when $n=3$ \cite{curvature2018}. In the physical setup, $k\in\mathbb{R}_+$ is the wavenumber and $q_1$ characterises the medium content of an inhomogeneity $D$. In nonlinear optics or acoustics \cite{Boyd}, $f(\mathbf x, u)$ can be of a more general form than that in \eqref{eq:phy1}, say e.g. {$f(\mathbf x, u)=k^2 q_1(\mathbf x)u+q_2(\mathbf x) u^2$ } to characterise the nonlinear effect. In a similar manner, \eqref{eq:helm1} can also be used to describe the Schr\"odinger equation that governs the quantum scattering (cf. \cite{HunSig}). On the other hand, we note that the well-posedness of the elliptic system \eqref{eq:helm1} has been extensively studied in the literature: in the linear case, the well-posedness is well understood \cite{LSSZ,McL}; and in the nonlinear case, the well-posedness can be achieved in many generic setups (cf. \cite{IMN} and the references cited therein) and in particular, if smallness is imposed on the solution, which in many situations of practical interest is equivalent to imposing smallness on the boundary input $\psi$, the well-posedness of \eqref{eq:helm1} can also be guaranteed; see e.g. \cite{LLLS} where the nonlinear term $f(x, u)$ is assumed to belong to a certain analytic class. Since our focus is the inverse problems \eqref{eq:ip1} and \eqref{eq:ip2}, and also in order to appeal for a general study, we always assume the well-posedness of the forward problem \eqref{eq:helm1}. Nevertheless, for self-containedness as well as our use, we establish the well-posedness for small solutions of the forward problem \eqref{eq:helm1} when $f(x, u)$ is only assumed to belong to the H\"older class. 

The inverse inclusion problem \eqref{eq:ip1} is a longstanding problem in the theory of inverse problems, but mainly restricted to linear mediums. We refer to \cite{moi5,moi6} for recent progress in electrostatics, \cite{BL3,curvature2018,Blasten2020,CDL,DCL, LT1} in inverse acoustic scattering, \cite{BLX2020,DFLY} in inverse electromagnetic scattering and \cite{BLin,DLS} in inverse elastic scattering. To our best knowledge, there is no result available for the inverse inclusion problem \eqref{eq:ip1} associated with general nonlinear anomalies. On the other hand, we note that recently there are many studies on the inverse boundary problem of recovering $f$ by knowledge of $\Lambda(\psi)$ associated with all $\psi\in H^{1/2}(\partial \Omega)$; that is, infinitely/uncountably many boundary measurements are needed. We refer to \cite{ImaYam,Isa1,IsaNac,Sun,LLLS} and the references cited therein for related results. It is pointed out that in all of those inverse problem studies, $f(\mathbf x, u)$ is usually required to possess higher regularities than the H\"older one required in the current article. By aiming at recovering the support of the anomaly, but not its physical content, we can work with merely H\"older continuous nonlinearities. Moreover, it is emphasised that we only make use of a single boundary measurement. If $f(x, u)$ is of a particular (still general) form, we can determine both the support and its physical content of the anomalous inclusion by a  minimal number of boundary measurements. Nevertheless, it is also pointed out that we require that $D$ is of polygonal/polyhedral or corona-shape since the corner or conic singularities are essentially needed in our mathematical argument. The mathematical arguments are based on microlally characterising the singularities in a quantitative manner of the solution $u$ to \eqref{eq:helm1} induced by the geometric singularities in $f(\mathbf x, u)$. Finally, we would like to emphasise that the results obtained in this paper include the relevant ones for linear mediums as special cases, and moreover our study indicates that the nonlinear effect can induce new phenomena that are of both theoretical and practical interest.

In summary, we list the major contributions of this work in what follows. 
\begin{enumerate}
\item We establish local and global uniqueness results in determining certain general nonlinear anomalies in several separate cases by minimal boundary measurements. These results are highly interesting, in particular in the following two aspects. First, to our best knowledge, this is first result in the literature concerning the shape determination of general nonlinear anomalies by a single measurement. The existing studies are mainly devoted to the determination of linear anomalies. Second, there are many existing studies on inverse problems for nonlinear differential equations, but most of them make use of infinitely many measurements. 

\item In achieving the results in (1), we need to impose ``strong" a-priori information on the target anomaly in that either its support or its physical content belongs to certain admissible classes. Nevertheless, on the one hand, these admissible classes are general enough to include some physically important cases, and on the other hand, they are good examples to verify that in the theory of inverse problems, the a-priori information can bring beneficial advantages to the inversion process. 

\item It is also worth noting that our results include many existing studies for linear anomalies as special cases. Moreover, they extend and generalise the relevant studies in that our results show that the nonlinearities can leverage certain technical restrictions in the linear counterpart and can help identify the anomalies; see Remark \ref{rem:26}  for more relevant discussion.

\end{enumerate}

The rest of the paper is organised as follows. In Section 2, we present the unique identifiability results for general anomalies including local uniqueness results with corner/conic singularities and a global unique result within polygonal/polyhedral or corona geometry. In Section 3, we present unique identifiability results for inverse problem \eqref{eq:ip2} with a single-layer structure. Section 4 is devoted to deriving unique identifiability results in determining layered anomalies.

\section{Determining supports of anomalous inclusions by a single measurement}

In this section, we consider the inverse boundary problem \eqref{eq:ip1} in determining the support of an anomalous inclusion independent of its physical content by a single boundary measurement.

\subsection{Local uniqueness results}

First, we introduce the geometric setup of our study. For a given point $\mathbf{x}_0\in\mathbb{R}^n$, $n=2,3$, we let $\mathbf{v}_0=\mathbf y_0-\mathbf x_0$ where $\mathbf y_0\in\mathbb{R}^n$ is fixed.  Set
       \begin{equation}\label{eq:cone1}
{\mathcal S}_{\mathbf{x}_0,\theta_0}:= \left\{\mathbf{y} \in \mathbb R^n ~|~0\leqslant \angle(\mathbf y-\mathbf{x}_0,\mathbf{v}_0)\leqslant \theta_0\right \}\ (\theta_0 \in(0,\pi/2)),
       \end{equation}
      which is a strictly convex conic cone with the apex $\mathbf x_0$ and an opening angle $2\theta_0 \in(0,\pi)$  in $\mathbb R^n$. Here $\mathbf v_0$ is referred to be the axis of $\mathcal C_{\mathbf x_0,\theta_0}$. Define the truncated conic cone as
       \begin{equation}\label{eq:cone2}
       \mathcal S^{h}_{\mathbf x_0,\theta_0}:=\mathcal S_{\mathbf x_0, \theta_0}\cap B_{h}(\mathbf x_0 ),
       \end{equation}
       where $B_{h}(\mathbf x_0)$ is an open ball centered at $\mathbf x_0$ with the radius $h\in \mathbb R_+$. When $n=2$, $\mathcal S^{h}_{\mathbf x_0}$ is a sectorial corner with the apex $\mathbf x_0$ and an opening angle $2\theta_0  \in (0,\pi)$. 
       

We also introduce a polyhedral corner in $\mathbb R^3$ as follows.  Assume that $\mathcal K_{\mathbf x_0;\mathbf e_1,\ldots, \mathbf e_\ell}$ is a polyhedral cone with the apex $\mathbf x_0$ and edges $\mathbf e_j$ ($j=1,\ldots, \ell$, $\ell\geq 3$), where $\mathbf e_j$, $j=1,2,\ldots \ell$ are mutually linearly independent vectors in $\mathbb{R}^3$. Throughout of this paper we always suppose that  $  \mathcal K_{\mathbf x_0;\mathbf e_1,\ldots, \mathbf e_\ell}$ is strictly convex, which implies that   it can be fitted  into a conic  cone $\mathcal S_{\mathbf x_0, \theta_0}$ with an opening angle $\theta_0\in (0,\pi/2)$, where $\mathcal S_{\mathbf x_0, \theta_0}$  is defined in \eqref{eq:cone1}.  Given a constant $h\in \mathbb R_+$, we define the truncated polyhedral corner $\mathcal K_{\mathbf x_0}^{h}$ as
     \begin{equation}\label{eq:kr0}
       \mathcal K_{\mathbf x_0}^{h}=\mathcal K_{\mathbf{x_0}; {\mathbf e_1},\ldots {\mathbf e_\ell }}\cap B_{h}(\mathbf x_0).
	\end{equation}


Throughout the rest of the paper, we denote 
\begin{equation}\label{eq:C_h def}
	\mathcal C_h:=  \mathcal S^{h}_{\mathbf x_0,\theta_0}\   \mbox{ or } \  \mathcal K^{h}_{\mathbf x_0}
\end{equation}
as a corner in $\mathbb R^n$ ($n=2,3$) with the apex $\mathbf x_0$, where $\mathcal S^{h}_{\mathbf x_0,\theta_0}$ and $\mathcal K^{h}_{\mathbf x_0}$ are defined in \eqref{eq:cone2} and \eqref{eq:kr0} respectively. The schematic illustration of a conic and polyhedral corner is displayed in Figure \ref{f2}. 


\begin{figure}[htbp]
	\centering
	\begin{minipage}{0.4\linewidth}
		\centering
		\includegraphics[width=0.35\linewidth]{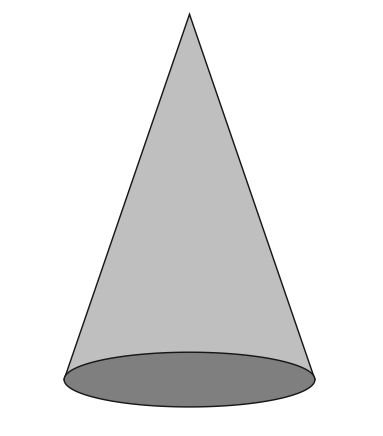}
	\end{minipage}
	\begin{minipage}{0.4\linewidth}
		\centering
		\includegraphics[width=0.4\linewidth]{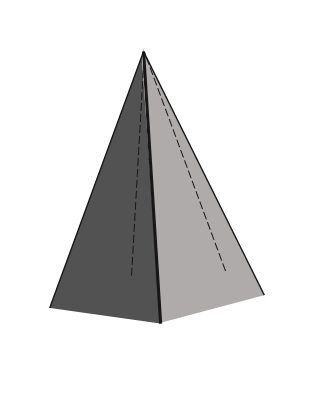}
	\end{minipage}
\caption{Illustrations of conic and polyhedral corner}\label{f2}
\end{figure}

\begin{lemma}
	\label{lem:cgoest}
Suppose that $\tau \in \mathbb R_+$ and $\mathcal C_h$ is defined in \eqref{eq:C_h def}. For $\mathbf x\in \mathbb R^n$ ($n=2,3$), let
\begin{equation}\label{eq:cgo}
u_0(\mathbf x)=e^{\tau(\mathbf d+\mathrm i  \mathbf d^\perp)\cdot \mathbf x},
\end{equation}
where $\mathbf d\cdot \mathbf d^\perp=0$ with $\mathbf d, \mathbf d^\perp \in \mathbb  S^{n-1}$,  then $\Delta u_0=0\ \mathrm{in}\  \mathbb R^n$. There exist unit vectors $\mathbf d,\ \mathbf d^\perp \in \mathbb S^{n-1}$ and   a positive number $\zeta$ depending on $\mathcal C_h$ satisfying 
\begin{equation}\label{eq:d cond}
	-1<\mathbf d\cdot \hat{\mathbf x}\leq -\zeta<0 \quad \mbox{for  all} \quad \mathbf x \in \mathcal C_h, \quad \mbox{and} \quad \mathbf d\cdot \mathbf d^\perp=0, 
\end{equation}
where $\hat{\mathbf x}=\frac{\mathbf x}{\vert \mathbf x\vert}$. 
Furthermore, for sufficient large $\tau$, it holds that
\begin{align}
\left \vert \int_{\mathcal C^h} u_0(\mathbf x)\mathrm d\mathbf x\right \vert & \geq  C_{\mathcal C_h} \tau^{-n}+\mathcal O\left(\tau^{-1}e^{-\frac{1}{2}\zeta h\tau}\right)\label{eq:conicest1},\\
\left\vert \int_{\mathcal C_h} \vert \mathbf x\vert ^\alpha u_0(\mathbf x)\mathrm d\mathbf x\right \vert &\lesssim \tau^{-(\alpha +n)}+\frac{1}{\tau}e^{-\frac{1}{2}\zeta h\tau},\quad \forall \alpha\in \mathbb R_+\label{eq:conicest2},\\
\|u_0\|_{H^1(\partial \mathcal C_h\cap \partial B_{h}(\mathbf x_0))}&\lesssim (2\tau^2 +1)^\frac{1}{2}e^{-\zeta h\tau},\label{eq:est 29} \\
\|\partial _\nu u_0\|_{L^2(\partial \mathcal C_h\cap \partial B_{h}(\mathbf x_0))}&\lesssim \sqrt 2\tau e^{-\zeta h\tau}, \label{eq:est 210}
\end{align}
where $ C_{\mathcal C_h} $ is a positive constant not depending on $\tau$. Here $``\lesssim"$ means that we neglect the generic constant $C$ associated with the principle  term  with respect to $\tau$ in the upper bounds of \eqref{eq:conicest2}, \eqref{eq:est 29} and  \eqref{eq:est 210} respectively, where $C$ does not depend on  $\tau$.  
\end{lemma}

\begin{proof}
Since $\mathbf d\perp \mathbf d^\perp$, one knows that $\Delta  u_0=0$ in $\mathbb R^n$. Without loss of generality, in the following we assume that {the apex $\mathbf x_0$ of $\mathcal C_h$} is the origin. In view of the convexity of $\mathcal C_h$ defined \eqref{eq:C_h def}, there exists a vector $\mathbf d\in \mathbb S^{n-1}$ satisfying  \eqref{eq:d cond}.  

	{In what follows, we only prove the cases that $\mathcal C_h$ is a sectorial corner in  $\mathbb R^2$ and $\mathcal C_h$ is a conic corner in $\mathbb R^3$ respectively. The case that $\mathcal C_h$ is a polyhedral corner in $\mathbb R^3$ can be proved similarly and we only remark it at the end of the proof.   }
	
For a fixed $\alpha  \in \mathbb R_+$, if $\Re{\mu}\geq 2\alpha/e$, where $\mu \in \mathbb C$, it yields that $r^\alpha\leq e^{\Re \mu r/2}$. Hence we have
	\begin{equation}\label{eq:int 29}
	\left | \int_{\varepsilon}^\infty r^\alpha e^ {-\mu r}\mathrm d\mathrm r\right |
	\leq \int_{\varepsilon}^\infty e^{-\Re{\mu}r/2}\mathrm d\mathrm r=\frac{2}{\Re{\mu}}e^{-\Re{\mu}\epsilon /2}.
	\end{equation}
	where  $\varepsilon \in \mathbb R_+$ is fixed. Using Laplace transform, one can derive that
	\begin{equation}\label{eq:estgam}
\int_{0}^\varepsilon r^\alpha e^{-\mu r}\mathrm d\mathrm r
=\frac{\Gamma(\alpha+1)}{\mu^{\alpha +1}}+\int_{\varepsilon}^\infty r^\alpha e^ {-\mu r}\mathrm d\mathrm r,
\end{equation}
	where $\Gamma$ is the Gamma function.
\medskip

{\noindent {\bf Case 1:} 	$\mathcal C_h$ is a sectorial  corner. }  Write  $\mathbf x=(x_1,x_2)\in\mathbb  R^2$  in the polar coordinates as $\mathbf x=(r\cos\theta,r\sin \theta )$, where $r\geq 0$ and $\theta \in [0,2\pi )$. Let $\Gamma_h^\pm$ be two edges  of  $\mathcal C_h$. Set 
$$
\Gamma_h^+=\{\mathbf x\in \mathbb R^2~|~\mathbf x=r(\cos \theta_M,\sin \theta_M )\},\quad \Gamma_h^-=\{\mathbf x\in \mathbb R^2~|~\mathbf x=r(\cos \theta_m,\sin \theta_m )\},
$$
where $r\in [0,h]$ with $h\in \mathbb R_+$, $\theta_m,\theta_M\in [0,2\pi)$ and $\theta_M-\theta_m=2\theta_0$. Here $2\theta_0$ is the opening angle  of $\mathcal C_h$, where $\theta_0\in (0,\pi/2)$.  Using the polar-coordinate transformation and \eqref{eq:estgam}, it can be obtained that
	\begin{equation}
	\begin{aligned}
	\int_{\mathcal C_h} e^{\rho \cdot \mathbf x}\mathrm d \mathbf x &= \int_{\mathcal C_h} e^{-\tau({\mathbf d}+ {\mathrm i \mathbf{d}^\perp)\cdot\mathbf {\hat x}}}\mathrm d \mathbf x 
	=\frac{\Gamma(2)}{\tau^2}\int_{\theta_{m}}^{\theta_M}\frac{1}{\left(\mathbf d\cdot \mathbf{\hat x}+\mathrm i \mathbf{d}^\perp \cdot \mathbf{\hat x} \right)^2}\mathrm d \theta-\int_{\theta_{m}}^{\theta_M}I_{r_1}\mathrm d\theta,
	\end{aligned}
	\end{equation}
	where $I_{r_1}= \int_{h}^{\infty}e^{-\tau (\mathbf d +\mathrm i\mathbf d)\cdot \hat{\mathbf x}r}r\mathrm d r$.
	Hence, it can be directly calculated that
	\begin{equation}
	\left\vert \int_{\theta_{m}}^{\theta_M}\frac{1}{\left(\mathbf d\cdot \mathbf{\hat x}+\mathrm i \mathbf{d}^\perp \cdot \mathbf{\hat x} \right)^2}\mathrm d \theta\right\vert =\frac{\theta_{M}-\theta_{m}}{\left| \mathbf d \cdot \mathbf{\hat x}(\theta_{\xi})+ \mathrm i \mathbf{d}^\perp \cdot \mathbf{\hat x} (\theta_{\xi})\right|^2} \geq \frac{\theta_M-\theta_m }{2}
	\end{equation}
	by using the integral mean value theorem.
	For sufficiently large $\tau$, according to \eqref{eq:d cond}  and \eqref{eq:int 29}, we have the following integral inequality
	\begin{equation}\label{1eq:eta0tau2}
	\begin{aligned}
	\left\vert \int_{\mathcal C_h} e^{\rho \cdot \mathbf x}\mathrm d \mathbf x  \right\vert
	&\geq \frac{\Gamma(2)(\theta_{M}-\theta_{m})}{2\tau^2} - \left\vert \int_{\theta_{m}}^{\theta_M} I_{\sf{R}} \mathrm d\mathbf \theta \right\vert
	\geq  \frac{C_{\mathcal C_h }}{\tau^2}-\frac{2}{\zeta\tau}e^{-\frac{1}{2}\zeta h \tau}.
	\end{aligned}
	\end{equation}
Therefore, we prove \eqref{eq:conicest1} for $n=2$,  where $C_{\mathcal C_h}=\theta_M-\theta_m=2\theta_0$ in \eqref{eq:conicest1}.

By adopting  a similar argument for  \eqref{eq:conicest1} when $\mathcal C_h$ is a sectorial corner, for \eqref{eq:conicest2} we have
	\begin{equation}\notag 
	\int_{\mathcal C_h}\vert \mathbf x\vert ^\alpha u_0\mathrm d\mathbf x=\frac{\Gamma(\alpha +2)}{\tau^{\alpha +2}}\int_{\theta_m}^{\theta_M}\left(\frac{1}{(\mathbf d\cdot \hat{\mathbf x}+\mathrm i\mathbf d^\perp\cdot \hat{\mathbf x})^2} +I_{R}\right) \mathrm d\theta,
	\end{equation}
	where we utilize \eqref{eq:estgam} and $
	I_R=\int_{h}^\infty r^{\alpha+2}e^{r(\tau \mathbf d\cdot \hat{\mathbf x}+\mathrm i\mathbf d^\perp\cdot \hat{\mathbf x})}\mathrm d r.
	$ Using \eqref{eq:int 29}, we can prove \eqref{eq:conicest2}.

By using the polar-coordinate transformation and \eqref{eq:d cond}, we have the following inequality:
\begin{align}\label{eq:214 u_0}
\|u_0\|_{L^2(\partial \mathcal C_h\cap \partial B_{h}(\mathbf 0))}
&=\left (\int_{\theta_m}^{\theta_M}e^{2h\tau \mathbf d\cdot \hat{\mathbf x}} \mathrm d \theta \right )^\frac{1}{2}\leq (\theta_M-\theta_m)^\frac{1}{2}e^{-\zeta h\tau}.
\end{align}
In view of \eqref{eq:cgo} and \eqref{eq:214 u_0}, one can directly verify that
\begin{equation}\notag
\begin{aligned}
\|u_0\|_{H^1(\partial \mathcal C_h\cap \partial B_{h}(\mathbf 0))}
&=\left(\|u_0\|^2_{L^{2}(\partial \mathcal C_h\cap \partial B_{h}(\mathbf 0))}+\|\tau(\mathbf d+\mathrm i\mathbf d^\perp) u_0\|^2_{L^{2}(\partial \mathcal C_h\cap \partial B_{h}(\mathbf 0))}\right)^{\frac{1}{2}}\\
&\leq (2\tau^2+1)^\frac{1}{2}\|u_0\|_{L^2(\partial \mathcal C_h\cap \partial B_{h}(\mathbf 0))}\lesssim (2\tau^2 +1)^\frac{1}{2}e^{-\zeta h\tau}, \ n=2,3. 
\end{aligned}
\end{equation}
Furthermore, by virtue  of \eqref{eq:d cond} and Cauchy-Schwarz inequality, it yields that
\begin{equation}\notag
\begin{aligned}
\|\partial _\nu u_0\|_{L^2(\partial \mathcal C_h\cap \partial B_{h}(\mathbf 0))}&
\leq \|\nabla u_0\|_{L^2(\partial \mathcal C_h\cap \partial B_{h}(\mathbf 0))}\lesssim \sqrt 2\tau e^{-\zeta h\tau}. 
\end{aligned}
\end{equation}
Therefore we obtain the estimates  \eqref{eq:est 29} and  \eqref{eq:est 210} as $\tau\rightarrow \infty$,  respectively. 
\medskip

{\noindent {\bf Case 2:} 	$\mathcal C_h$ is a conic corner in $\mathbb R^3$. } Recall that $\mathcal C_h$ has the opening angle $2\theta_0$, which is defined in \eqref{eq:cone1}.  By virtue of \eqref{eq:estgam}, it yields that
	\begin{equation}\notag
	\int_{\mathcal C^h} e^{\tau (\mathbf d+\mathrm i \mathbf d^\perp)\cdot \mathbf x}\mathrm d\mathbf x=I_1+  \int_{0}^{2\pi}\mathrm d\varphi\int_{0}^{\theta_0}I_{r_2}\sin \theta\mathrm d\theta,
	\end{equation}
	where 
	$$
	I_1=\int_{0}^{2\pi}\int_{0}^{\theta_0}\left(\frac{\Gamma(3)}{\tau^3(\mathbf d\cdot \hat{\mathbf x}+\mathrm i\mathbf d^\perp \cdot \hat{\mathbf x})^3}\right)\sin \theta\mathrm d\varphi\mathrm d\theta,\quad I_{r_2}=\int_{h}^\infty r^2e^{r\tau( \mathbf d\cdot \hat{\mathbf x}+\mathrm i\mathbf d^\perp \cdot \hat{\mathbf x})}\mathrm d r.
	$$
By the
 integral mean value theorem, one can arrive at that
	\begin{equation}\label{eq:I1 214}
		\begin{aligned}
	I_1
	&=\frac{\Gamma(3)}{\tau^3}\int_{0}^{2\pi}\frac{1}{\left(\mathbf d\cdot \hat{\mathbf x}(\varphi,\theta_\xi)+\mathrm i\mathbf d^\perp \cdot \hat{\mathbf x}(\varphi,\theta_\xi)\right)^3}\mathrm d\varphi\int_{0}^{\theta_0}\sin \theta_0\mathrm d\theta\\
	&=\frac{2\pi\Gamma(3)(1-\cos \theta_0)}{\tau^3}\frac{1}{\left(\mathbf d\cdot \hat{\mathbf x}(\varphi_\xi,\theta_\xi)+\mathrm i\mathbf d^\perp \cdot \hat{\mathbf x}(\varphi_\xi,\theta_\xi)\right)^3}.
	\end{aligned}
	\end{equation}
For sufficient large $\tau$, from \eqref{eq:int 29}, one has
	\begin{equation}\notag
	\vert I_{r_2}\vert =\left|\int_{h}^\infty r^2e^{r\tau( \mathbf d\cdot \hat{\mathbf x}+\mathrm i\mathbf d^\perp \cdot \hat{\mathbf x})}\mathrm d \mathrm r\right|\leq \frac{2}{\zeta\tau}e^{-\frac{1}{2}\zeta h\tau},
	\end{equation}
which implies that
	\begin{equation}\label{eq:Ir2}
	\left \vert \int_{0}^{2\pi}\mathrm d\varphi \int_{0}^{\theta_0}I_{r_2}\sin\theta\mathrm d\theta\right\vert \leq\int_{0}^{2\pi}\int_{0}^{\theta_0}\vert I_{r_2}\vert \mathrm d\varphi\mathrm d\theta
	\leq\frac{4\pi \theta_0}{\zeta\tau}e^{-\frac{1}{2}\zeta h\tau}. 
	\end{equation}
	From \eqref{eq:I1 214} and \eqref{eq:Ir2},  using  Cauchy-Schwarz inequality and \eqref{eq:d cond}, one can prove \eqref{eq:conicest1} for the case that  $\mathcal C_h$ is a conic corner, where $C_{\mathcal C_h}=\sqrt{2}\pi(1-\cos\theta_0)$ in \eqref{eq:conicest1} .

For \eqref{eq:conicest2}, from \eqref{eq:estgam}, it yields that
	\begin{equation}\notag
	\begin{aligned}
	\int_{\mathcal C^h}| \mathbf x|^\alpha u_0\mathrm d\mathbf x&=\frac{\Gamma(n+3)}{\tau^{n+3}}\int_{0}^{2\pi}\mathrm d\varphi\int_{0}^{\theta_0}\left(\frac{1}{(\mathbf d\cdot \hat{\mathbf x}+\mathrm i\mathbf d^\perp\cdot \hat{\mathbf x})^3} +I_R\right)\sin\theta\mathrm d\theta,
	\end{aligned}
	\end{equation}
	where
	$
	I_R=\int_{h}^\infty r^{\alpha+3}e^{r(\tau \mathbf d\cdot \hat{\mathbf x}+\mathrm i\mathbf d^\perp\cdot \hat{\mathbf x})}\mathrm d r	$. In view  of \eqref{eq:int 29} we obtain \eqref{eq:conicest2}.
	
	For \eqref{eq:est 29}, according to  polar coordinate transformation and \eqref{eq:d cond}, one has
\begin{align}\notag
\|u_0\|_{L^2(\partial \mathcal C_h\cap \partial B_{h}(\mathbf 0))}&=\left (\int_{0}^{\theta_0}\int_{0}^{2\pi}e^{2h\tau \mathbf d\cdot \hat{\mathbf x}}\mathrm d\varphi\mathrm d\theta\right )^\frac{1}{2}
\leq (2\pi \theta_0)^\frac{1}{2}e^{-\zeta h\tau},
\end{align}
which  can be  used to derive \eqref{eq:est 29} immediately. Similarly, using  \eqref{eq:d cond} and Cauchy-Schwarz inequality, one can  show that \eqref{eq:est 210} is valid  for the case that $\mathcal  C_h$ is a conic corner. 

Finally, the case that $\mathcal C_h$ is a polyhedral corner in $\mathbb R^3$ can be proved in a similar manner; see also \cite[Lemma 3.4]{BL3}. 
	
	The proof is complete. 
\end{proof}

A main auxiliary theorem is given as follows.

\begin{theorem}\label{thm:1}
Let $(D; f)\in\mathscr{A}$ and $\mathcal{C}_h$ be a corner. Consider the following system of differential equations for $u\in H_{loc}^2(\mathcal{C}_h)$ and $v\in H_{loc}^2(\mathcal{C}_h)$:
\begin{equation}\label{eq:a1}
\begin{cases}
\Delta u+f(\mathbf x, u)=0 &\quad \mbox{in}\ \ \mathcal{C}_h,\\
\Delta v+\lambda v=0 & \quad\mbox{in}\ \ \mathcal{C}_h,\\
u=v,\quad \partial_\nu u=\partial_\nu v & \quad\mbox{on}\ \ \partial\mathcal{C}_h\backslash \partial B_h ,
\end{cases}
\end{equation}
where $\nu$ is the exterior unit normal vector to $\partial D$. Then one has
\begin{equation}\label{eq:a2}
 \lambda u(\mathbf  x_ 0)-f(\mathbf x_0, u(\mathbf x_0))=0,
\end{equation}
where $\mathbf x_0$ is the apex of $\mathcal{C}_h$.
\end{theorem}

\begin{proof} 
Since $\Delta$ is invariant under rigid motions, without loss of generality, we assume that the apex $\mathbf x_0$ of $\mathcal C_h$ coincides with the origin.  By virtue of Green's formula and \eqref{eq:a1}, we have the following integral identity: 
\begin{equation}\label{eq:cgogreen}
\int_{\mathcal C_h} (\lambda v-f(\mathbf x,u))u_0\mathrm d \mathbf x =\int_{\partial C_h\cap \partial B_h(\mathbf 0) }u_0\partial_\nu(u-v)-(u-v)\partial_\nu u_0\mathrm d\sigma,
\end{equation}
where $u_0$ is defined in \eqref{eq:cgo}. 
According to Sobolev's embedding theorem,  we have  $u,v\in C^\beta (\mathcal C_h)$ ($\beta\in (0,1]$ for $n=2$ and $\beta\in (0,1/2]$ for $n=3$) since $u,v\in H^2(\mathcal C_h)$. By further using the H\"older continuity of $f(x,\cdot)$ and the transmission conditions, we can derive the following expansions:
\begin{align}
F(\mathbf x):=\lambda v-f(\mathbf x,u)&=\lambda u(\mathbf 0)-f(\mathbf 0,u(\mathbf 0))+\delta_{v}(\mathbf x)+\delta_{f(\mathbf x,u)}(\mathbf x)  \notag,\\
 \vert \delta_{f(\mathbf x,u)}(\mathbf x) \vert &\leq \| f(\mathbf x,u)\|_{C^\alpha(\mathcal C_h)}\vert \mathbf x\vert ^\alpha,\quad \vert \delta_{v}(\mathbf x) \vert \leq \| v\|_{C^\alpha(\mathcal C_h)}\vert \mathbf x\vert ^\alpha  \label{eq:alphaex},
\end{align}
where $\alpha=\min\{\beta,\gamma\} \in (0, 1)$ depending on the H\"older indices $\gamma$ and $\beta$. 

Combining \eqref{eq:alphaex} with \eqref{eq:cgogreen}, one can show that
\begin{align}
(\lambda u (\mathbf 0)-f(\mathbf 0,u(\mathbf 0)))\int_{\mathcal C_h}u_0(\mathbf x)\mathrm d\mathbf x
&=-\int_{\mathcal C_h}(\delta_v(\mathbf x)+\delta_{f(\mathbf x,u)}) u_0\mathrm d\mathbf x \notag\\
&+\int_{\partial C_h\cap \partial B_h(\mathbf 0)}u_0\partial_\nu(u-v)-(u-v)\partial_\nu u_0\mathrm d\sigma \label{eq:I}.
\end{align}
By virtue of \eqref{eq:alphaex} and \eqref{eq:conicest2}, one has
\begin{equation}
\begin{aligned}
\left\vert \int_{\mathcal C_h}(\delta_v(\mathbf x)+\delta_{f(\mathbf x,u)}) u_0\mathrm d\mathbf x \right\vert
&\leq (\|v\|_{C^\alpha(\mathcal C_h)}+\|f\|_{C^\alpha(\mathcal C_h)})\int_{\mathcal C_h}\vert \mathbf x\vert^\alpha |u_0|\mathrm d\mathbf x\\
&\lesssim \tau^{-(\alpha +n)}+\frac{1}{\tau}e^{-\frac{1}{2}\zeta h\tau},\ n=2,3.
\end{aligned}
\end{equation}

According to  the trace theorem and the fact that $u,v\in H^1(\mathcal C_h)$, from \eqref{eq:est 29} and \eqref{eq:est 210},  we can deduce that 
\begin{align}
\left \vert\int_{\partial C_h\cap \partial B_h(\mathbf 0)}u_0\partial_\nu(u-v)\mathrm d\sigma \right \vert &\leq \|u_{0}\|_{H^{\frac{1}{2}}(\partial C_h\cap \partial B_h(\mathbf 0))} \|\partial_{\nu}(u-v)\|_{H^{-\frac{1}{2}}(\partial C_h\cap\partial B_h(\mathbf 0))}\notag
\\
&\leq C\|u_0\|_{H^1(\partial \mathcal C_h\cap\partial B_h(\mathbf 0))}\|u-v\|_{H^1(C_h)}\notag\\
&\lesssim (2\tau^2 +1)^\frac{1}{2}e^{-\zeta h\tau},\label{eq:J1}\\
\left \vert \int_{\partial C_h\cap\partial B_h(\mathbf 0)}(u-v)\partial_\nu u_0\mathrm d\sigma \right \vert &\leq \|\partial _{\nu} u_0\|_{ L^{2}(\partial C_h\cap \partial B_h(\mathbf 0))} \|u-v\|_{ L^{2}(\partial C_h\cap \partial B_h(\mathbf 0))}\notag \\
&\leq C\|u-v\|_{H^1(\mathcal C_h)}\|\partial _\nu u_0\|_{L^2(\partial \mathcal C_h\cap \partial B_h(\mathbf 0))},
\notag\\
&\leq\sqrt 2C \tau e^{-\zeta h\tau},\label{eq:J2}
\end{align}
as $\tau \rightarrow \infty$, where $C$ is a generic constant originating from the trace theorem.

 Substituting \eqref{eq:conicest1}, \eqref{eq:conicest2}, \eqref{eq:J1} and \eqref{eq:J2} into \eqref{eq:I}, one has
 \begin{equation}\label{eq:u0f0}
 \left( C_{\mathcal C_h} \tau^{-n} +\mathcal O\left(\tau^{-1}e^{-\frac{1}{2}\zeta h\tau} \right) \right)\vert \lambda u(\mathbf 0)-f(\mathbf 0,u(\mathbf 0))\vert \lesssim \tau^{-(\alpha +n)}+(1+\tau)e^{-\zeta h\tau}+\frac{1}{\tau}e^{-\frac{1}{2}\zeta h\tau}
 \end{equation}
as $\tau\to \infty$. Multiplying $\tau^n$ on both sides of \eqref{eq:u0f0} and letting $\tau \to \infty$, then we can derive \eqref{eq:a2}. We complete the proof of Theorem \ref{thm:1}.
\end{proof}

We can show a local unique recovery result for the inverse problem \eqref{eq:ip1}. Before that, we introduce an admissibility condition for $\psi$.

\medskip

\noindent{\bf Assumption A.}~We say that $\psi\in H^{1/2}(\partial\Omega)$ is admissible and write $\psi\in \mathscr{B}$ if the solution to \eqref{eq:helm1} fulfills:
\begin{equation}\label{eq:a3}
\lambda u(\mathbf x_c)-f(\mathbf x, u(\mathbf x_c))\neq 0,\quad \mbox{or} \quad \lambda u(\mathbf x)-f(\mathbf x, u(\mathbf x))\neq 0,\quad \forall \mathbf x\in \Omega \backslash \overline{D}, 
\end{equation}
where $\mathbf x_c \in \partial D$ satisfies  $D\cap B_h(\mathbf x_c)=\mathcal C_h$ defined in \eqref{eq:C_h def} for a sufficient small $h\in \mathbb R_+$. 

It is emphasised that in Section~\ref{sect:5}, we shall show that Assumption A can hold in a certain generic scenario of practical interest. 

\begin{theorem}\label{thm:2}
Let $(D_j; f_j)\in\mathscr{A}$, $j=1,2$, and suppose that
\begin{equation}\label{eq:a4}
\Lambda_{D_1, f_1}(\psi)=\Lambda_{D_2, f_2}(\psi)\quad\mbox{for a fixed $\psi\in \mathscr{B}$}.
\end{equation}
Then $D_1\Delta D_2$ cannot possess a corner on $\partial \mathbf{G}$, where $\mathbf{G}$ is the connected component of $\Omega\backslash\overline{D_1\cup D_2}$ that connects to $\partial \Omega$.
\end{theorem}

\begin{proof}
By contradiction and also noting that $\Delta$ is invariant under rigid motion,  without loss of generality, we assume that there exists a corner $\mathcal C_h$ defined \eqref{eq:C_h def} satisfying $ D_2\cap B_h(\mathbf 0)=\mathcal C_h \Subset \Omega \backslash \overline{D_1 }$, where $\mathbf 0 \in \partial D_2$. Let $u_j$ be the wave field to the scattering problem \eqref{eq:helm1} associated with $D_j$, $j=1,2$.  By virtue of \eqref{eq:a4}, using the fact that $u_j$ is real analytic in $\Omega \backslash \overline {( D_1\cup D_2)}$, from unique continuation principle,  it yields that  
	\begin{equation}\label{eq:th23}
\begin{cases}
\Delta u_2+f_2(\mathbf x, u_2)=0 &\quad \mbox{in}\ \ \mathcal{C}_h,\\
\Delta u_1+\lambda u_1=0 & \quad\mbox{in}\ \ \mathcal{C}_h,\\
u_2=u_1,\quad \partial_\nu u_2=\partial_\nu u_1 & \quad\mbox{on}\ \ \partial\mathcal{C}_h\backslash \partial B_h. 
\end{cases}
\end{equation}
Since $f_2\in \mathscr{A}$, according to Theorem \ref{thm:1}, it arrives that
$
 \lambda u_2(\mathbf 0)-f_2(\mathbf 0, u(\mathbf 0))=0,
$
which contradicts to \eqref{eq:a3}.

The proof is complete. 
\end{proof}

\subsection{Global unique identifiability results}

If we impose certain a-prior knowledge on the inclusion, we can establish the global uniqueness in determining the shape of the inclusion by a single measurement in the following two theorems by utilizing Theorem \ref{thm:2} and contradiction  arguments. 

\begin{theorem}\label{thm:3}
Let $(D; f)\in\mathscr{A}$, where $D$ is a convex polygon in $\mathbb{R}^2$ or a convex polyhedron in $\mathbb R^3$. Then $D$ is uniquely determined by a single boundary measurement $\Lambda_{D, f}(\psi)$ with a fixed $\psi \in \mathscr{B}$. 

\end{theorem}

\begin{figure}
\label{fig:3 cor}
	\centering
	\begin{minipage}{0.49\linewidth}
		\centering
		\includegraphics[width=0.5\linewidth]{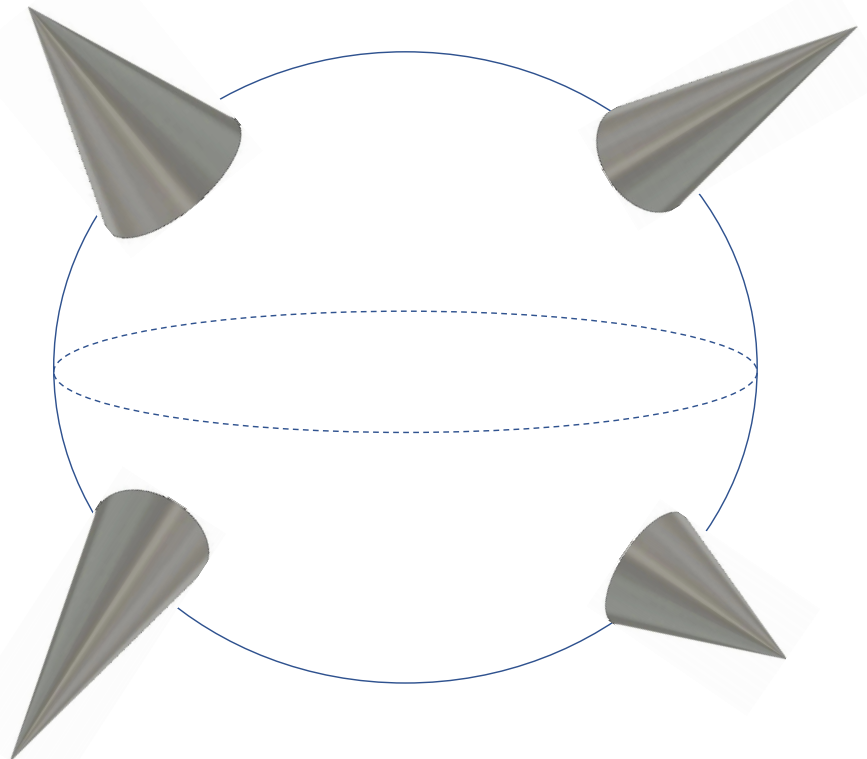}
	\end{minipage}
	\begin{minipage}{0.49\linewidth}
		\centering
		\includegraphics[width=0.5\linewidth]{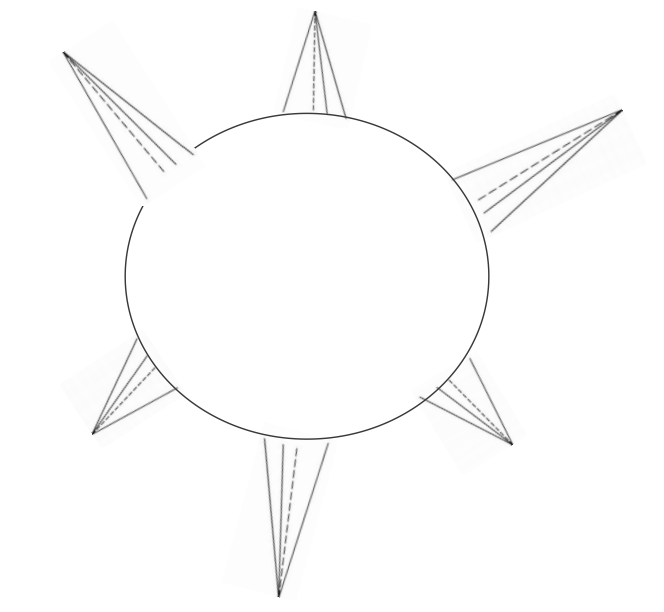}
	\end{minipage}
	
	\begin{minipage}{0.49\linewidth}
		\centering
		\includegraphics[width=0.5\linewidth]{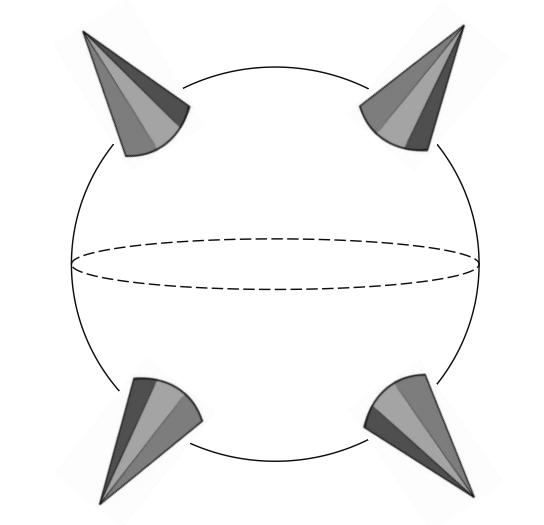}
	\end{minipage}
	\begin{minipage}{0.49\linewidth}
		\centering
		\includegraphics[width=0.5\linewidth]{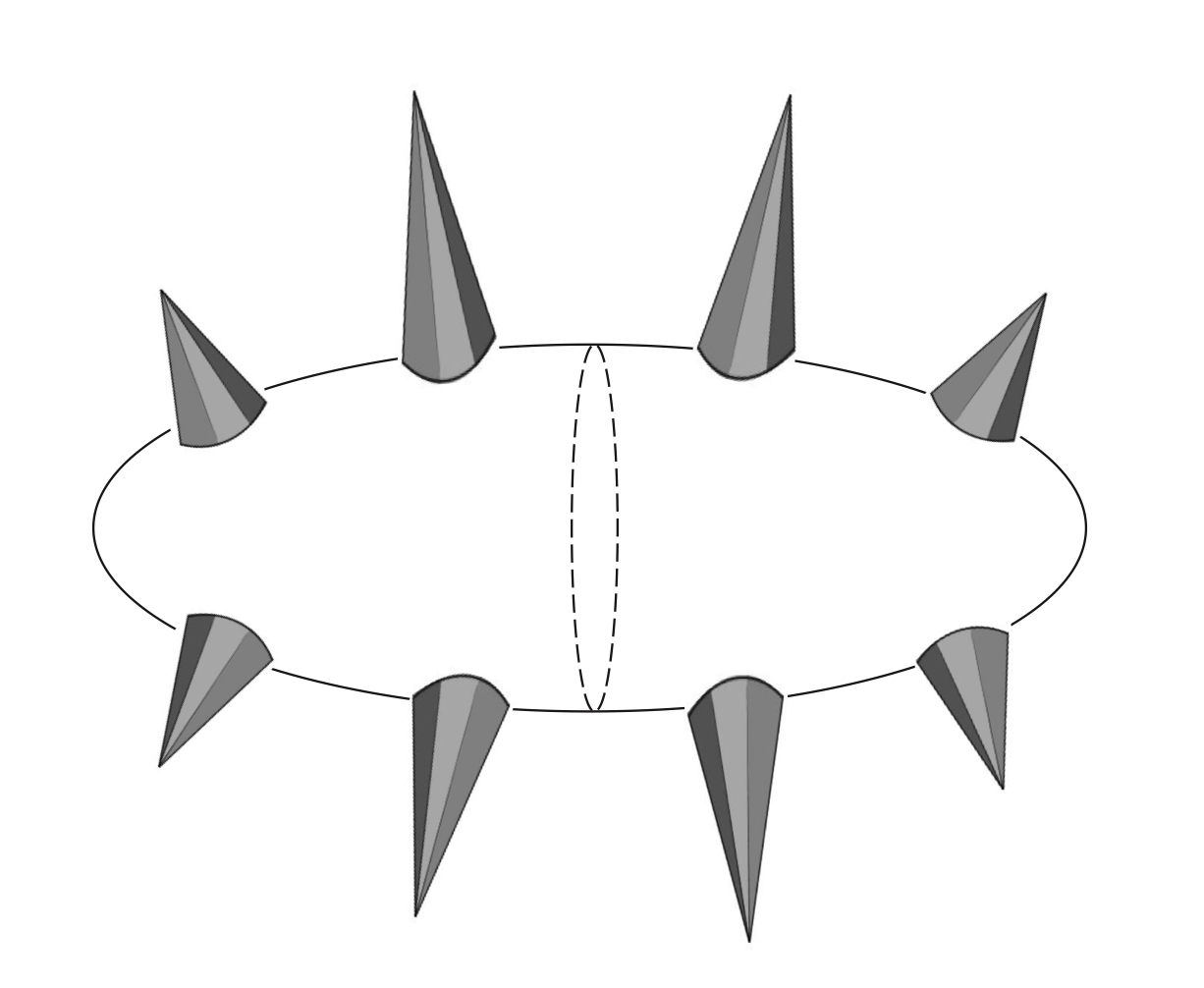}
	\end{minipage}
	\caption{Schematic illustration of corona-shape scatterers. }
\end{figure}

In the following we introduce an admissible class $\mathcal T$ of corona shapes, which shall be used in Theorem \ref{thm:finiteconic}. The schematic illustration of corona-shape  scatterers is displayed in Figure 3. 

\begin{defn}\label{def:cone3}
	Let $\widetilde D$ be a convex bounded Lipschitz domain with a connected complement $\mathbb R^3\setminus\overline{D}$. If there exsit finitely many strictly convex conic cones $\mathcal S_{\mathbf x_j,\theta_j}(j=1,2,\dots,\ell,\ell\in \mathbb N)$ defined in \eqref{eq:cone1} such that
	\begin{itemize}
		\item [(a)]
		the apex $\mathbf x_j\in \mathbb R^3\setminus \overline{\widetilde D}$ and  let $\mathcal S_{\mathbf x_j,\theta_j}^{\ast}=\mathcal S_{\mathbf x_j,\theta_j}\setminus \overline{ \widetilde D} $ respectively, where the apex $\mathbf x_j$ belongs to the strictly convex bounded conic corner of $\mathcal S_{\mathbf x_j,\theta_j}^{\ast}$;
		\item [(b)]
		$\partial \overline{\mathcal S_{\mathbf x_j,\theta_j}^{\ast}}\setminus \partial \overline{\mathcal S_{\mathbf x_j,\theta_j}}\subset \partial \overline {\widetilde D}$ and $\cap _{j=1}^\ell\partial \overline{\mathcal S_{\mathbf x_j,\theta_j}^{\ast}}\setminus \partial \overline{\mathcal S_{\mathbf x_j,\theta_j}}=\emptyset$;
	\end{itemize}
	then $ D :=\cup_{j=1}^\ell\mathcal S_{\mathbf x_j,\theta_j} \cup \widetilde D $ is said to belong to a class $\mathcal T$ of corona shape. 	
\end{defn}

\begin{theorem}\label{thm:finiteconic}
	Suppose that $  D_{m},m=1,2$  belong to the admissible class $\mathcal T$ of corona shape, where
	\begin{equation}\label{eq:Dm}
		D_m=\cup_{j^{(m)}=1}^{\ell^{(m)}}\mathcal C_{\mathbf x_{j^{(m)}},\theta_{j^{(m)}}}\cup 	\widetilde {D_m},\quad m=1,2.
	\end{equation}
	Consider the  scattering problem \eqref{eq:helm1} associated with   $( D _{m},f_m) \in \mathscr{A} ,m=1,2$. If the following conditions:
	\begin{subequations}
		\begin{align}
		\Lambda_{ D_1, f_1}(\psi)&=\Lambda_{ D_2, f_2}(\psi)\quad\mbox{for a fixed $\psi\in \mathscr{B}$}\label{eq:thm25} \\
				\widetilde { D_1}&=	\widetilde {D_2}, \label{eq:ass 56a} \\
			 \theta_{i^{(1)}}&=\theta_{j^{(2)}} \   \mbox{for} \  i^{(1)} \in \{1,\ldots,\ell^{(1)}\}\  \mbox{and}  \  j^{(2)}\in \{1,\ldots,\ell^{(2)}
		\}\ \mbox{when} \ \mathbf x_{i^{(1)}}=\mathbf x_{j^{(2)}},  \label{eq:ass 56b}
		\end{align}
	\end{subequations}
	then $\ell^{(1)}=\ell^{(2)},\ \mathbf x_{j^{(1)}}=\mathbf x_{j^{(2)}}$ and $\theta_{j^{(1)}}=\theta_{j^{(2)}}$, where $j^{(m)}=1,\dots \ell^{(m)}$, $m=1,2$. Namely, one has $ D_1=D _2$. 	
\end{theorem}

\begin{proof}
	We prove this theorem by contradiction. Suppose that $ D_1 \neq D _2$, due to \eqref{eq:ass 56a}  and  \eqref{eq:ass 56b}, without loss of generality  one concludes that there exists a conic corner $\mathcal S^h_{\mathbf x_c,\theta_c } \subset  D_2 \backslash \overline{ D_1 }$. Under \eqref{eq:thm25}, by virtue of Theorem \ref{thm:2}, we get the contradiction. 
\end{proof}

\begin{remark}\label{rem:26}
In Theorems \ref{thm:3} and \ref{thm:finiteconic},  a single boundary measurement $\Lambda_{ D, f}(\psi)$ can uniquely determine the inclusion $D$ under certain a-prior knowledge on $D$, where $\psi\in \mathscr{B}$. Namely, if $\psi\in \mathscr{B}$, the the admissible condition $\lambda u(\mathbf x_c)-f(\mathbf x, u(\mathbf x_c))\neq 0$ is fulfilled, where $\mathbf x_c$ is an apex of $D$ and $u$ is the  solution to \eqref{eq:helm1} associated  with $\psi$. The aforementioned admissible condition covers the corresponding admissible assumption for previous uniquely shape  determination of a  convex polygonal or polyhedral or corona-shape acoustic medium scatter $D$ by a single far-field measurement in inverse acoustic scattering problems (cf.\cite[Theorem 4.1]{DCL}) and \cite[Theorems 5.2, 5.3  and Corollary 5.5]{DFL}, where the medium parameter $f(\mathbf x, u)$ characterizing $D$ is linear with respect to the total wave field $u$, namely $f(\mathbf x, u)=q u$ with $q\in L^\infty(D )$. Indeed, the admissible assumption in  \cite{DCL,DFL} is $(q(\mathbf x_c)-\lambda)u(\mathbf x_c) \neq 0$, where $q$ is H\"older continuous near the neighborhood of $\mathbf x_c$. On the other hand, the nonlinearities can leverage certain technical restrictions in the linear counterpart and can help identify the anomalies. For example, when $f(\mathbf x, u)=\lambda u+q(\mathbf x)u^2$, where $f$ has the same linear term as the background medium configuration,  the admissible condition \eqref{eq:a3} turns out to be $q(\mathbf x_c) u^2(\mathbf x_c)\neq 0$. Therefore, for this specific form of  $f(\mathbf x, u)$ characterizing  the anomalous inclusion $D$, although the linear term in $f(\mathbf x, u)$  cannot contribute to the shape determination of $D$,  the nonlinear term  in $f(\mathbf x, u)$ can help one to identity $D$ by a single boundary measurement $\Lambda_{ D, f}(\psi)$ under the admissible condition $q(\mathbf x_c) u^2(\mathbf x_c)\neq 0$, where $D$  is a convex polygon or polyhedron or corona-shape inclusion with certain a-prior knowledge described in Theorem \ref{thm:finiteconic}. 

\end{remark}

\section{Determining both supports and contents of anomalous inclusions}

In this section, we consider the inverse boundary problem \eqref{eq:ip2} in determining both the support and its physical content of an anomalous inclusion by a minimal number of boundary measurements. Throughout the present section, we consider $a(\mathbf x, u)$ in \eqref{eq:helm1} of the following form:
\begin{equation}\label{eq:form1}
a(\mathbf x, u)=\bigg(\sum_{j=1}^N \lambda_j u^j-\lambda u\bigg)\chi_D+\lambda u\chi_\Omega,\quad \mathbf x\in \Omega,
\end{equation} 
where $\lambda_j\in\mathbb{C}$. That is, the inhomogeneity inside $D$ is given by
\begin{equation}\label{eq:form2}
f(\mathbf x, u)=\sum_{j=1}^N \lambda_j u^j,\quad \lambda_j\in\mathbb{C}. 
\end{equation}
Next, we shall show that an anomalous inclusion of the form $(D; f)$ with $D$ being a convex polygon/polyhedron or an admissible corona shape and $f$ of the form \eqref{eq:form2} can be uniquely determined uniquely determined by $N$ properly chosen boundary measurements. To that end, we introduce the following admissibility condition on the boundary inputs. 

\medskip

 \noindent{\bf Assumption B.}~Let $(D; f)$ be described above. We say that $\psi_j\in H^{1/2}(\partial\Omega)$, $j=1,2,\ldots, N$, are admissible and write $\psi_j\in \mathscr{H}$ if the corresponding solutions to \eqref{eq:helm1}, written as $u_{\psi_j}$ in what follows, fulfil the following condition:
\begin{equation}\label{eq:aa3}
\lambda u_{\psi_j}(\mathbf x_c)-f(\mathbf x_c, u_{\psi_j}(\mathbf x_c))\neq 0,\ \ 1\leq j\leq N; \qquad \prod_{1\leq i\leq j\leq N} \left(u_{\psi_j} (\mathbf{x}_c)-u_{\psi_i}(\mathbf{x}_c)\right)\neq 0,
\end{equation}
where $\mathbf x_c \in \partial D$ satisfies  $D\cap B_h(\mathbf x_c)=\mathcal C_h$ defined in \eqref{eq:C_h def} for a sufficient small $h\in \mathbb R_+$. 

\medskip

Similar to Assumption A, we shall show in Section~\ref{sect:5} that Assumption B can hold in a certain generic scenario of practical interest.

 \begin{theorem}\label{thm:simul1}
Let $(D; f)\in\mathscr{A}$,  where $D$ is a convex polygon in $\mathbb{R}^2$ or a convex polyhedron in $\mathbb R^3$. Assume that $f$ is of the form \eqref{eq:form2}. Then both $D$ and $f$ are uniquely determined by $N$ boundary measurements $\Lambda_{D, f}(\psi_j)$ with $\psi_j \in \mathscr{H}$, $j=1,2,\ldots, N$. 

Assume that $D_m$, $m=1,2$, are two  an admissible corona shape as described in Definition~\ref{def:cone3}, where $D_m$ is defined  by \eqref{eq:Dm}. Suppose that
\begin{equation}
	f_m(\mathbf x, u)=\sum_{j=1}^{N_m} \lambda_{j,(m)} u^j,\quad \lambda_{j,(m)}\in\mathbb{C},\quad N_m\in\mathbb N. 
\end{equation}
If the assumption \eqref{eq:ass 56a}, \eqref{eq:ass 56b} and
\begin{equation}
	\Lambda_{ D_1, f_1}(\psi_j)=\Lambda_{ D_2, f_2}(\psi_j)\quad\mbox{for a fixed $\psi\in \mathscr{H}$},\quad j=1,\ldots,\max\{N_1,N_2\}. 
\end{equation}
are fulfilled, then $D_1=D_2$, $N:=N_1=N_2$ and $\lambda_{j,(1)}=\lambda_{j,(2)}  $, $j=1,\ldots, N$. 
\end{theorem}

In order to prove Theorem~\ref{thm:simul1}, we first derive an auxiliary lemma. 

\begin{lemma}\label{lem:aux1}
Let $f_m\in\mathscr{A}$, $m=1,2$, and $\mathcal{C}_h$ be a corner. Consider the following system of differential equations for $u\in H_{loc}^2(\mathcal{C}_h)$ and $v\in H_{loc}^2(\mathcal{C}_h)$:
\begin{equation}\label{eq:lem32 eq}
\begin{cases}
\Delta u+f_1(\mathbf x, u)=0 &\quad \mbox{in}\ \ \mathcal{C}_h,\\
\Delta v+f_2(\mathbf x, v)=0 & \quad\mbox{in}\ \ \mathcal{C}_h,\\
u=v,\quad \partial_\nu u=\partial_\nu v & \quad\mbox{on}\ \ \partial\mathcal{C}_h\backslash \partial B_h ,
\end{cases}
\end{equation}
where $\nu$ is the exterior unit normal vector to $\partial D$. Then one has
\begin{equation}\label{eq:lem 32 a2}
 f_1(\mathbf{x}_0, u(\mathbf{x}_0))-f_2(\mathbf{x}_0, v(\mathbf x_0))=0,
\end{equation}
where $\mathbf x_0$ is the apex of $\mathcal{C}_h$.
\end{lemma}

\begin{proof} 
The proof of this lemma is similar to that for Theorem \ref{thm:1}. We sketch the argument in what follows. Let $u_0$ be given \eqref{eq:cgo} by letting $\lambda=0$. By virtue of \eqref{eq:lem32 eq} and Green's formula, one has
\begin{equation}\label{eq:lemma 32 cgogreen}
\int_{\mathcal C_h} (f_1(\mathbf x,u)-f_2(\mathbf x,v))u_0\mathrm d \mathbf x =\int_{\partial C_h\backslash \partial B_h(\mathbf x_0) }u_0\partial_\nu(u-v)-(u-v)\partial_\nu u_0\mathrm d\sigma. 
\end{equation}
By Sobolev's embedding theorem,  we have  $u,v\in C^\beta (\mathcal C_h)$ ($\beta\in (0,1]$ for $n=2$ and $\beta\in (0,1/2]$ for $n=3$) since $u,v\in H^2(\mathcal C_h)$. By further using the H\"older continuity of $f_m(\mathbf{x},\cdot)$, $m=1,2$, it yields that following expansions:
\begin{align}
F(\mathbf x):=f_1(\mathbf x,u)-f_2(\mathbf x,v)&=f_1(\mathbf x_0,u(\mathbf x_0))-f_2(\mathbf x_0,v(\mathbf x_0))+\delta_{f_1(\mathbf x,u)-f_2(\mathbf x,v)}(\mathbf x)\notag,\\
 \vert \delta _{f_1(\mathbf x,u)-f_2(\mathbf x,v)}(\mathbf x)\vert &\leq \| F(\mathbf x)\|_{C^\alpha(\mathcal C_h)}\vert \mathbf x\vert ^\alpha \label{eq:lem 32 alphaex},
\end{align}
where $\alpha\in (0, 1)$ depending on the H\"older indices $\gamma$ and $\beta$. 

In view of \eqref{eq:lemma 32 cgogreen} and \eqref{eq:lem 32 alphaex}, by virtue of \eqref{eq:conicest1}, we can follow the similar argument in the proof of Theorem \ref{thm:1} to prove this lemma. 
\end{proof}

\begin{proof}[Proof of Theorem~\ref{thm:simul1}]
Let $(D_m; f_m)$, $m=1,2$, be two anomalous inclusions as described in the statement of the theorem. Assume that
\begin{equation}\label{eq:pp1}
f_m(\mathbf x, u)=\sum_{j=1}^{N_m} \lambda_j^{(m)} u^j,\quad \lambda^{(m)}_j\in\mathbb{C},\ \ m=1,2.
\end{equation}
By introducing zero coefficients if necessary, we can assume that $N_1=N_2$ and set $N:=N_1=N_2$. 
We also assume that
\begin{equation}\label{eq:pp2}
\Lambda_{D_1, f_1}(\psi_j)=\Lambda_{D_2, f_2}(\psi_j), \ \ \psi_j\in\mathscr{H}, \ \ j=1,2,\ldots,N. 
\end{equation}

First, by following a similar argument to the proofs of Theorems~\ref{thm:3} and \ref{thm:finiteconic}, and using the first admissibility condition in \eqref{eq:aa3}, one can show that
\begin{equation}\label{eq:aa4}
D_1=D_2. 
\end{equation}

Set $D=D_1=D_2$ and let $\mathcal{C}_h$ be a corner on $\partial D$ with the apex being $\mathbf{x}_0$. By \eqref{eq:pp2}, we have
\begin{equation}\label{eq:aa5}
\begin{cases}
\Delta u^{(1)}_{\psi_j}+f_1(\mathbf x, u^{(1)}_{\psi_j})=0 &\quad \mbox{in}\ \ \mathcal{C}_h,\\
\Delta u^{(2)}_{\psi_j}+f_2(\mathbf x, u^{(2)}_{\psi_j})=0 & \quad\mbox{in}\ \ \mathcal{C}_h,\\
u_{\psi_j}^{(1)}=u_{\psi_j}^{(2)},\quad \partial_\nu u_{\psi_j}^{(1)}=\partial_\nu u_{\psi_j}^{(2)} & \quad\mbox{on}\ \ \partial\mathcal{C}_h\backslash \partial B_h ,
\end{cases}
\end{equation}
for $j=1,2,\ldots, N$, where $u_{\psi_j}^{(m)}$ signifies the solution to \eqref{eq:helm1} associated with $f_m$ and $\psi_j$, $m=1, 2$ and $1\leq j\leq N$. By Lemma~\ref{lem:aux1}, we readily have
\begin{equation}\label{eq:aa6}
\sum_{j=1}^N \lambda_j^{(1)} [u_{\psi_i}^{(1)}(\mathbf{x}_0)]^j-\sum_{j=1}^N \lambda_j^{(2)} [u_{\psi_i}^{(2)}(\mathbf{x}_0)]^j=0,\quad i=1,2,\ldots, N. 
\end{equation}
On the other hand, by \eqref{eq:aa5}, we note that
\begin{equation}\label{eq:aa7}
u_{\psi_i}^{(1)}(\mathbf{x}_0)=u_{\psi_i}^{(2)}(\mathbf{x}_0):=u_{\psi_i}(\mathbf{x}_0),\quad i=1,2,\ldots, N. 
\end{equation}
By combining \eqref{eq:aa6} and \eqref{eq:aa7}, we readily have
\begin{equation}\label{eq:aa8}
\sum_{j=1}^N \big(\lambda_j^{(1)}-\lambda_j^{(2)} \big) [u_{\psi_i}(\mathbf{x}_0)]^j=0,
\end{equation}
which together with the second admissibility condition in \eqref{eq:aa3} readily yields that
\[
\lambda_j^{(1)}=\lambda_j^{(2)},\quad j=1,2,\ldots, N. 
\]

The unique determination for the support and its physical content   of an admissible inclusion $(D;f)$ of corona shape as described in Definition~\ref{def:cone3} by $N$ measurements can be proved in a similar way, where $N$ is an a-prior parameter of $f$ with the form \eqref{eq:form2}.

The proof is complete. 
\end{proof}

\section{Determining embeded nonlinear anomalies}

In this  section  we consider the determination of the shape and physical parameters of the embedded nonlinear anomalies by minimal  measurements,  which have a polygonal or polyhedral nest structure. We first introduce several definitions. 

\begin{defn}\label{def:1}
$D$ is said to have a polygonal-nest  or polyhedral-nest partition if there exist $\Sigma_\ell$, $\ell=1, 2, \ldots, N$, $N\in \mathbb{N}$, such that each $\Sigma_\ell$ is an open convex simply-connected polygon or polyhedron and
\begin{equation}\label{eq:def1}
\Sigma_N\Subset\Sigma_{N-1}\Subset\cdots\Subset\Sigma_2\Subset\Sigma_1=D. 
\end{equation}
\end{defn}

\begin{figure}[h]
	\centerline{\includegraphics[width=0.4\linewidth]{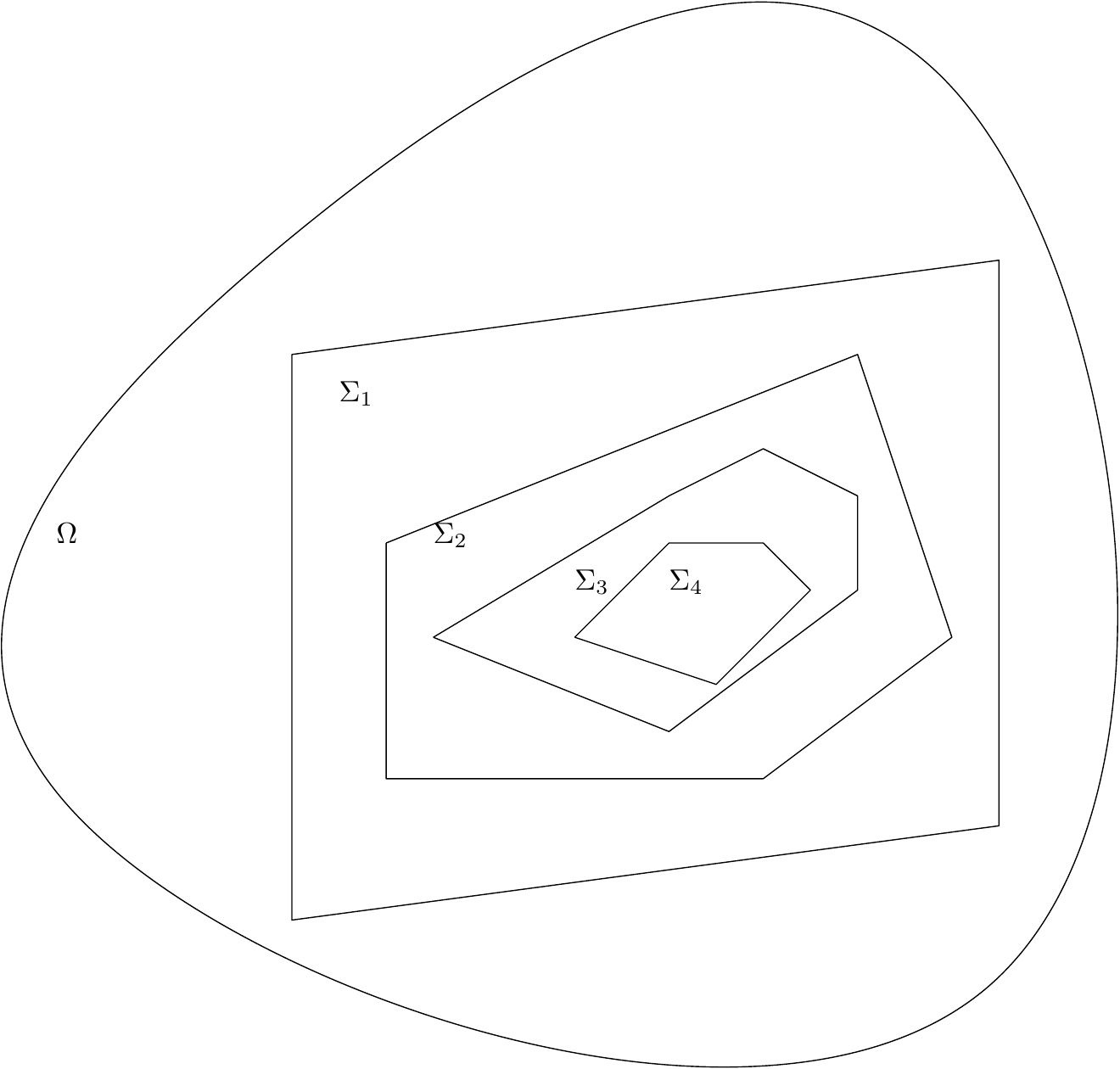}}
	\caption{Schematic  illustration  of the polygonal-nest structure.}
\end{figure}

In the follow two definitions, we introduce an anomalous inclusion possessing  a polygonal-nest or polyhedral-nest structure of the class $\mathcal A$ or $\mathcal B$, respectively.

\begin{defn}\label{def:3}
Let $(D; f)\in\mathscr{A}$ be an anomalous inclusion. It is said to possess a polygonal-nest or polyhedral-nest structure of the class $\mathcal A$ if the following conditions are fulfilled:
\begin{enumerate}
\item $D$ has a polygonal-nest or polyhedral-nest partition as described in Definition~\ref{def:1};

\item each $\Sigma_\ell$ is an anomalous inclusion such that
\begin{equation}
\begin{split}
\label{eq:f sec4}
	f(\mathbf x ,u)\big|_{ U_\ell}&=\sum_{j=1}^{M_{\ell} } \lambda_j^{(\ell)} u^j,\quad \lambda_j\in\mathbb{C},\quad U_\ell:=\Sigma_\ell\backslash\overline{\Sigma_{\ell+1}},\quad M_\ell \in \mathbb N, \quad 1\leq \ell\leq N,
\end{split}
\end{equation}
where for any $\ell \in \{1,\ldots,N-1\}$, it holds that
\begin{align}\notag
	\sum_{j=1}^{M_{\ell} } \lambda_j^{(\ell)} t^j\neq \sum_{j=1}^{M_{\ell+1} } \lambda_j^{(\ell+1)} t^j,\quad t\in \mathbb C. 
\end{align}

\end{enumerate}

\end{defn}

\begin{defn}\label{def:4}
Let $(D; f)\in\mathscr{A}$ be an anomalous inclusion. It is said to possess a polygonal-nest or polyhedral-nest structure of the class $\mathcal B$ if the following conditions are fulfilled:
\begin{enumerate}
\item $D$ has a polygonal-nest or polyhedral-nest partition as described in Definition~\ref{def:1};

\item each $\Sigma_\ell$ is an anomalous inclusion such that
\begin{equation}
\begin{split}
\label{eq:f sec5}
	f(\mathbf x ,u)\big|_{ U_\ell}&=\lambda_\ell  u,\quad \lambda_\ell \in\mathbb{C},\quad U_\ell:=\Sigma_\ell\backslash\overline{\Sigma_{\ell+1}}, \quad 1\leq \ell\leq N-1, \\  
	f(\mathbf x ,u)\big|_{ \Sigma_N}&=f_N(\mathbf x, u),\quad f_N\in \mathscr{A},
\end{split}
\end{equation}
where 
\begin{equation}\label{eq:lambda  cond}
	\lambda_\ell \neq \lambda_{\ell+1},\quad \ell=1,\ldots, N-2. 
\end{equation}


\end{enumerate}

\end{defn}

We shall give the unique shape  and physical parameter determination for two admissible classes introduced in Theorems \ref{thm:41} and \ref{cor:41} by minimal boundary measurements under the following admissible assumption. 

\medskip

 \noindent{\bf Assumption C.}~Let $(D; f)$ be described in  Definition \ref{def:3}, where $f$ has the form \eqref{eq:f sec4}. We say that $\psi_j^{(\ell)}\in H^{1/2}(\partial\Omega)$, $j=1,2,\ldots, M_{\ell}$, $\ell=1,\ldots,N$, are admissible and write $\psi_j^{(\ell)}\in \mathscr{C}$ if the corresponding solutions to \eqref{eq:helm1}, written as $u_{\psi_j^{(\ell)}}$ in what follows, fulfill the following condition:
\begin{equation}\label{eq:aa sec 4}
\begin{split}
&\lambda u_{\psi_j^{(1)}}(\mathbf y_c)-f(\mathbf y_c, u_{\psi_j^{(1)}}(\mathbf y_c)) \big|_{U_{1}}\neq 0,\ \ 1\leq j\leq M_1,\quad  \mathbf \forall \mathbf y_c\in \mathcal{V}( \partial \Sigma_1),  \\
&\sum_{m=1}^{M_\ell}\lambda_m^{(\ell-1)}  u^m_{\psi_j^{(\ell  )}}(\mathbf x_c)-f(\mathbf x_c, u_{\psi_j^{(\ell)}}(\mathbf x_c))\big|_{U_{\ell}}\neq 0,\  1\leq j\leq M_{\ell}, \ \forall  \mathbf{x}_c\in \mathcal{V}( \partial \Sigma_\ell),\  \ell=2,\ldots,N-1, \\
& \prod_{1\leq i\leq j\leq M_\ell } \left(u_{\psi_j^{(\ell)}}(\mathbf{x}_c)-u_{\psi_i^{(\ell)}}(\mathbf{x}_c)\right)\neq 0,\quad \forall  \mathbf{x}_c\in \mathcal{V}( \partial \Sigma_\ell),	\quad  \ell=1,\ldots,N, 
\end{split}
\end{equation}
where $\mathcal{V}( \partial \Sigma_\ell)$ is the vertex set of $\Sigma_\ell$, $\ell=1,\ldots,N$.


\medskip

Similar to Assumptions A and B before, we shall show in Section~\ref{sect:5} that Assumption C can hold in a certain generic scenario of practical interest.

We are now in a position to present the main  theorem of this section. 

\begin{theorem}\label{thm:41}
Let $(D; f)\in\mathscr{A}$,  where $D$ has a  polygonal-nest structure in $\mathbb{R}^2$ or polyhedral-nest structure  in $\mathbb R^3$ of the class $\mathcal A$. Assume that $f$ is of  the form \eqref{eq:f sec4}. Then both $D$ and $f$ are uniquely determined by $\sum_{\ell=1}^N M_\ell$ boundary measurements $\Lambda_{D, f}(\psi_j^{(\ell)})$ with $\psi_j^{(\ell)} \in \mathscr{C}$, $j=1,2,\ldots, M_\ell $  and $\ell=1,\ldots, N$. 
\end{theorem}

\begin{proof}
 Assume that $(D_m; f_m)$ ($m=1,2$) are two anomalous inclusions as described in the statement of the theorem. Namely, 
\begin{equation}\label{eq:two d}
\begin{split}
&\Sigma_{N_1,(m)} \Subset\Sigma_{N_1-1,(m)}\Subset\cdots\Subset\Sigma_{2,(m)}\Subset\Sigma_{1,(m)}=D_m,\\
&f_m(\mathbf x ,u)\big|_{ U_{\ell,(m)}}=\sum_{j=1}^{M_{\ell,(m)} } \lambda_{j,(m)}^{(\ell)} u^j,\, \lambda_{j,(m)}^{(\ell)}\in\mathbb{C},\, M_{\ell,(m)} \in \mathbb N, \, 1\leq \ell\leq N_m, 
\end{split}
\end{equation}
where each $\Sigma_{\ell,(m)}$ is an open convex simply-connected polygon or polyhedron, and $ U_{\ell,(m)}:=\Sigma_{\ell,(m)}\backslash\overline{\Sigma_{\ell+1,(m)}}$.  Without loss of  generality we  assume that $N_1\leq N_2$. By  introducing zero coefficients if necessary, in view of \eqref{eq:two d}, one can readily know that 
\begin{equation}\label{eq:def1}
\begin{split}
&f_m(\mathbf x ,u)\big|_{ U_{\ell,(m)}}=\sum_{j=1}^{M_{\ell} } \lambda_{j,(m)}^{(\ell)} u^j,\, \lambda_{j,(m)}^{(\ell)}\in\mathbb{C}, \, 1\leq \ell\leq N_1, 
\end{split}
\end{equation}
where $M_{\ell} =\max\{M_{\ell,(1)}, M_{\ell,(2)}  \}$. 

Suppose that
\begin{align}\label{eq:psi assum}
	\Lambda_{D_1, f_1}(\psi_j^{(\ell)})=\Lambda_{D_2, f_2}(\psi_j^{(\ell)})\quad\mbox{for $\psi_j^{(\ell)} \in \mathscr{C}$},\quad j=1,\ldots,M_\ell,\quad \ell=1,\ldots, N_1. 
\end{align}

In the following we prove this theorem by mathematical induction. 
Under the assumption \eqref{eq:psi assum}, according to Theorem \ref{thm:simul1}, it holds  $\partial D_1=\partial D_2$,  which implies that $\partial \Sigma_{1,(1)}=\partial \Sigma_{1,(2)}$. Furthermore,  from Theorem \ref{thm:simul1}, one can claim that 
$$
\lambda_{j,(1)}^{(1)}=\lambda_{j,(2)}^{(1)},\quad j=1,2,\ldots, M_{1}. 
$$
Let $u^{(m)}_{\psi_j^{(\ell)}}$ be the solution  of \eqref{eq:helm1} associated with $(D_m;f_m)$ and $\psi_j^{(\ell)}$. Hence by unique continuation, one has
$$
u^{(1)}_{\psi_j^{(2) } }  \Big|_{U_1 } =u^{(2)}_{\psi_j^{(2) } }  \Big|_{U_1 } \mbox{  in } U_1=\Sigma_1\backslash \overline{\Sigma_2 }.
$$

Suppose that there exits an index $n_* \in \mathbb{N} \backslash\{1\} $ such that 
\begin{equation}\label{eq:lambda 48}
	\partial  \Sigma_{\ell}:= \partial \Sigma_{\ell,(1)}=\partial \Sigma_{\ell,(2)},   \quad \lambda_{j,(1)}^{(\ell)}=\lambda_{j,(2)}^{(\ell)},\quad j=1,\ldots,M_\ell, \quad {\ell=2,\ldots,n_*-1. }
\end{equation}
Therefore we can recursively prove that
\begin{equation}\label{eq:u1l+1}
\begin{split}
	&u^{(1)}_{\psi_j^{(\ell+1) } }  \Big|_{U_\ell   } =u^{(2)}_{\psi_j^{(\ell+1) } }  \Big|_{U_\ell } \mbox{  in } U_\ell =\Sigma_\ell \backslash \overline{\Sigma_{\ell+1} },\ j=1,2,\ldots, M_{\ell +1},  \ \ell=1,2,\ldots, n_*-2\\
	&u^{(1)}_{\psi_j^{(n_*) } }  \Big|_{U_{n_*-1,(m)}   } =u^{(2)}_{\psi_j^{(n_*) } }  \Big|_{U_{n_*-1,(m)} } \mbox{  in } U_{n_*-1,(m)} =\Sigma_{n_*-1} \backslash \overline{\Sigma_{n_*,(m)} },\, j=1,\ldots, M_{n_*},  \, m=1,2, 
	\end{split}
\end{equation}
by using \eqref{eq:psi assum}

Assume that $\partial \Sigma_{n_*,(1)}\neq \partial \Sigma_{n_*,(2)}$. By the convexity  of $\Sigma_{n_*,(m)}$ ($m=1,2$), without loss of generality, we can suppose that  there  exists a polyhedral corner $\mathcal C_h$ with the apex $\mathbf x_c$ satisfying $\mathcal C_h \subset  \Sigma_{n_*,(1)}\backslash\overline{\Sigma_{n_*,(2)} } $. According to \eqref{eq:lambda 48} and  \eqref{eq:u1l+1}, it yields that
\begin{equation} \notag 
\begin{cases}
\Delta u^{(1)}_{\psi_i^{(n_*)}}+\sum_{j=1}^{M_{n_*} } \lambda_{j,(1)}^{(n_*)} \left(u^{(1)}_{\psi_i^{(n_*)}}\right)^j=0 &\quad \mbox{in}\ \ \mathcal{C}_h,\\
\Delta u^{(2)}_{\psi_i^{(n_*)}}+\sum_{j=1}^{M_{n_*} } \lambda_{j,(1)}^{(n_*-1)} \left(u^{(2)}_{\psi_i^{(n_*)}}\right)^j=0 & \quad\mbox{in}\ \ \mathcal{C}_h,\\
u_{\psi_i^{(n_*)}}^{(1)}=u_{\psi_i^{(n_*)}}^{(2)},\quad \partial_\nu u_{\psi_i^{(n_*)}}^{(1)}=\partial_\nu u_{\psi_i^{(n_*)}}^{(2)} & \quad\mbox{on}\ \ \partial\mathcal{C}_h\backslash \partial B_h ,
\end{cases}
\end{equation}
where $u_{\psi_i^{(n_*)}}^{(m)}\in H^2(\mathcal C_h)$ by noting interior elliptic regularity. Using Lemma \ref{lem:aux1}, one  has
$$
\sum_{j=1}^{M_{n_*} } \lambda_{j,(1)}^{(n_*-1)} \left(u^{(1)}_{\psi_i^{(n_*)}}\right)^j- \sum_{j=1}^{M_{n_*} } \lambda_{j,(1)}^{(n_*)} \left(u^{(1)}_{\psi_i^{(n_*)}}\right)^j=0
$$
which contradicts  to the second admissible condition in \eqref{eq:aa sec 4}. Therefore, it is ready to know that $\partial \Sigma_{n_*}:=\partial \Sigma_{n_*,(1)}=\partial \Sigma_{n_*,(2)}$.

Let $\mathcal{C}_h$ be a corner on $\partial \Sigma_{n_*}=\partial \Sigma_{n_*,(1)}=\partial \Sigma_{n_*,(2)}  $ with the apex  $\mathbf{x}_0$. According to \eqref{eq:u1l+1}, it yields that
\begin{equation}\label{eq:aa5 s4}
\begin{cases}
\Delta u^{(1)}_{\psi_i^{(n_*)}}+\sum_{j=1}^{M_{n_*} } \lambda_{j,(1)}^{(n_*)} \left(u^{(1)}_{\psi_i^{(n_*)}}\right)^j=0 &\quad \mbox{in}\ \ \mathcal{C}_h,\\
\Delta u^{(2)}_{\psi_i^{(n_*)}}+\sum_{j=1}^{M_{n_*} } \lambda_{j,(2)}^{(n_*)} \left(u^{(2)}_{\psi_i^{(n_*)}}\right)^j=0 & \quad\mbox{in}\ \ \mathcal{C}_h,\\
u_{\psi_i^{(n_*)}}^{(1)}=u_{\psi_i^{(n_*)}}^{(2)},\quad \partial_\nu u_{\psi_i^{(n_*)}}^{(1)}=\partial_\nu u_{\psi_i^{(n_*)}}^{(2)} & \quad\mbox{on}\ \ \partial\mathcal{C}_h\backslash \partial B_h ,
\end{cases}
\end{equation}
for $i=1,2,\ldots, M_{n_*}$, where $u_{\psi_i^{(n_*)}}^{(m)}$ signifies the solution to \eqref{eq:helm1} associated with $f_m$ and $\psi_i^{(n_*)}$,  and $1\leq i\leq M_{n_*}$. Using Lemma~\ref{lem:aux1}, one readily has
\begin{equation}\label{eq:aa6m s4}
\sum_{j=1}^{M_{n_*}}\left( \lambda_{j,(1)}^{(n_*)}- \lambda_{j,(2)}^{(n_*)} \right) [u_{\psi_i^{(n_*)}}^{(1)}(\mathbf{x}_0)]^j=0,\quad i=1,2,\ldots, M_{n_*}. 
\end{equation}
by noting the transmission condition in  \eqref{eq:aa5 s4}. By virtue of the third admissible condition in \eqref{eq:aa sec 4} together with \eqref{eq:aa6m s4}, it arrives that 
\[
\lambda_{j,{(1)}}^{(n_*)}=\lambda_{j,{(2)}}^{(n_*)},\quad j=1,2,\ldots, M_{n_*}. 
\]
Moreover,  by \eqref{eq:psi assum} one conclude that
$$
u^{(1)}_{\psi_j^{(n_*+1) } }  \Big|_{U_{n_*, (m)}  } =u^{(2)}_{\psi_j^{(n_*+1) } }  \Big|_{U_{n_*,(m)} } \mbox{  in } U_{n_*,(m)} =\Sigma_{n_*} \backslash \overline{\Sigma_{n_*+1, (m)} }, \, j=1,\ldots, M_{n_*+1}, \, m=1,2. 
$$

 We can prove $N_1=N_2$ by using the contradiction. Indeed, we assume that $N_1<N_2$. Therefore,  there exits a corner point $\mathbf x_c \in \partial \Sigma_{N_1+1,(2)}$ lying inside of  $\Sigma_{N_1,(2)}$. From Lemma \ref{lem:aux1}, we can prove that
 $$
 \sum_{j=1}^{M_{N_1}}\lambda_j^{(N_1)} \left( u^{(2)}_{\psi_i^{(N_1  )}}(\mathbf x_c)\right)^j-f(\mathbf x_c, u^{(2)}_{\psi_i^{(N_1)}}(\mathbf x_c))\big|_{U_{N_1+1,(2)}}= 0,\ \ 1\leq i\leq M_{N_1},
  $$
 which contradicts to the second admissible condition of \eqref{eq:aa sec 4}.

The proof is complete. 
\end{proof}

In the next theorem we prove that an anomalous inclusion possessing a polygonal-nest structure in $\mathbb{R}^2$ or polyhedral-nest structure  in $\mathbb R^3$ of the class $\mathcal B$ can be uniquely determined  by a single boundary measurement fulfilling Assumption D introduced below.

\medskip

 \noindent{\bf Assumption D.}~Let $(D; f)$ be described in  Definition \ref{def:4}, where $f$ has the form \eqref{eq:f sec5}. We say that $\psi \in H^{1/2}(\partial\Omega)$ is admissible and write $\psi \in \mathscr{D}$ if the corresponding solutions to \eqref{eq:helm1}, written as $u_{\psi}$ in what follows, fulfill the following condition:
\begin{equation}\label{eq:aa sec 5}
\begin{split}
&\lambda u_{\psi }(\mathbf y_c)-f(\mathbf y_c, u_{\psi }(\mathbf y_c)) \big|_{U_{1}} \neq 0, \quad \forall \mathbf y_c \in\mathcal{V}( \partial \Sigma_1),    \\
& \lambda_{\ell-1}  u_{\psi }(\mathbf x_c)-f(\mathbf x_c, u_{\psi}(\mathbf x_c))\big|_{U_{\ell}}\neq 0, \quad \forall \mathbf x_c \in \mathcal{V}(\Sigma_{\ell}  ),  \quad  \ell=2,\ldots,N-1, \\
& \lambda_\ell  u_{\psi }(\mathbf x_c)-f_N(\mathbf x_c, u_{\psi}(\mathbf x_c)) \neq 0, \quad \forall  \mathbf{x}_c\in \mathcal{V}( \partial \Sigma_N),  \\
& u_{\psi }(\mathbf{x}_c) \neq 0, \quad \forall  \mathbf{x}_c\in \mathcal{V}( \partial \Sigma_\ell),\quad \ell=1,\ldots,N,
\end{split}
\end{equation}
where $\mathcal{V}( \partial \Sigma_\ell)$ is the vertex set of $\Sigma_\ell$, $\ell=1,\ldots,N.$ 

\medskip

In Section~\ref{sect:5}, we shall show that in a certain generic scenario of practical interest that Assumption D can hold.

\begin{theorem}\label{cor:41}
Let $(D; f)\in\mathscr{A}$,  where $D$ has a  polygonal-nest structure in $\mathbb{R}^2$ or polyhedral-nest structure  in $\mathbb R^3$ of the class $\mathcal B$. Assume that $f$ is of  the form \eqref{eq:f sec5}. Then both $D$ and the  physical parameters of $f$ $\lambda_\ell$ ($\ell=1,\ldots,N-1$) are uniquely determined by a single boundary measurement $\Lambda_{D, f}(\psi)$ with $\psi \in \mathscr{D}$. 
\end{theorem}

\begin{proof}
	We sketch the proof of this theorem by modifying necessary parts of the proof  of Theorem \ref{thm:41}. By contradiction, suppose that there exist two anomalous inclusions $(D_m;f_m)$ described by the statement of this theorem such that 
	\begin{align}\label{eq:psi assum nonlinear}
	\Lambda_{D_1, f_1}(\psi )=\Lambda_{D_2, f_2}(\psi )\quad\mbox{for $\psi  \in \mathscr{C}$}, 
\end{align} 
	where 
	\begin{equation}\label{eq:two d nonlinear}
\begin{split}
&\Sigma_{N_1,(m)} \Subset\Sigma_{N_1-1,(m)}\Subset\cdots\Subset\Sigma_{2,(m)}\Subset\Sigma_{1,(m)}=D_m,\\
&f_m(\mathbf x ,u)\big|_{ U_{\ell,(m)}}=\lambda_{\ell,(m)} u,\quad \lambda_{\ell ,(m)}\in\mathbb{C},\quad 1\leq \ell\leq N_m. 
\end{split}
\end{equation}
Here each $\Sigma_{\ell,(m)}$ is an open convex simply-connected polygon or polyhedron, and $ U_{\ell,(m)}:=\Sigma_{\ell,(m)}\backslash\overline{\Sigma_{\ell+1,(m)}}$.

Using the first admissible condition in \eqref{eq:aa sec 5}, under \eqref{eq:psi assum nonlinear},  from Theorem \ref{thm:simul1}, we can obtain that $\partial \Sigma_1=\partial \Sigma_2$. Once the  unique shape determination of $\partial \Sigma_1$ is derived, by using \eqref{lem:aux1} and noting the fourth admissible condition in \eqref{eq:aa sec 5},  we can prove that  $\lambda_{1,(1)}= \lambda_{1,(2)} $.   Therefore, one has $u^{(1)}_{\psi }  \big|_{U_1 } =u^{(2)}_{\psi}  \big|_{U_1 } \mbox{  in } U_1=\Sigma_1\backslash \overline{\Sigma_2 }$ by unique continuation principle, where $u^{(m)}_{\psi}$ is the solution to \eqref{eq:helm1} associated  with $(D_m;f_m)$ and $\psi$, $m=1,2$.

Following a similar argument  in the proof  of Theorem \ref{thm:41}, by virtue  of Lemma \ref{lem:aux1} and \eqref{eq:aa sec 5} we can prove that
$$
N_1=N_2:=N,\quad \Sigma_{\ell,(1)}=\Sigma_{\ell,(2)},\quad   \lambda_{\ell,(1)}=\lambda_{\ell,(2)},\quad \ell=1,\ldots,N.
$$

The proof is complete. 
\end{proof}

\section{Discussion on admissibility conditions}\label{sect:5}
In this  section we shall show that the technical {Assumptions A}, {B}, { C} and {D} introduced in the previous sections can be fulfilled under generic scenarios. 

Recall that $\Omega$ is a bounded Lipschitz domain in $\mathbb R^n$ ($n=2,3$), and $D$ is a   bounded Lipschitz domain such that $D\Subset \Omega$ and $\Omega \backslash \overline D$ is connected.  For illustrative purpose, we consider some specific nonlinear Helmholtz equations that arise in the time-harmonic wave scattering theory; see also our discussion in Introduction. It is emphasised that one can derive similar results in other setups by following similar arguments as discussed in what follows.

Let $u\in H^1(\Omega )$ be the solution to 
\begin{equation}\label{eq:sec 5}
	\begin{cases}
	\Delta u + a(\mathbf{x} , u) = 0 & \quad \mbox{ in }\quad \Omega,\\
	u = \psi(\varepsilon,k,\mathbf d ) & \quad \mbox{ on }\quad \partial\Omega,
	\end{cases}
	\end{equation}
where $a(\mathbf{x} , u)=(f(\mathbf x, u)-k^2 u)\chi_D+k^2 u $ with $k\in \mathbb R_+\cup\{0\}$ and $\psi(\mathbf x; \varepsilon,k,\mathbf d ):=\varepsilon e^{\mathrm i k \mathbf x\cdot \mathbf d}$ with $\mathbf d\in \mathbb S^{n-1}$ and $
\varepsilon \in \mathbb R_+ $ with $\varepsilon \ll 1$. 

In the following  proposition, when $f(\mathbf x, u)$ takes the form \eqref{eq:form2}, we shall prove that the solution $u$ to \eqref{eq:sec 5} can be decomposed as $u=\psi(\varepsilon,k,\mathbf d ) +v  $, where  $v$ can be viewed as a  small  perturbation.

\begin{proposition}\label{pro:51}
 Consider the semilinear elliptic equation \eqref{eq:sec 5}, where 
	$$
	f(\mathbf x, u)=\sum_{j=1}^N \lambda_j u^j,\quad \lambda_j\in\mathbb{C}. 
	$$
	Suppose that $u \in H^{1}(\Omega )$ is the solution to   \eqref{eq:sec 5}, which  satisfies $u=\psi(\mathbf x; \varepsilon,k,\mathbf d ) +v$. Then $v \in H^1(\Omega )$ fulfills
	\begin{equation}\label{eq:sec 5 v}
	\begin{cases}
	\Delta v + k^2 v \chi_{\Omega \backslash D } =( k^2 \psi -\sum\limits_{j=1}^N  \lambda_j u^j)\chi_D & \quad \mbox{ in }\quad \Omega,\\
	v = 0 & \quad \mbox{ on }\quad \partial\Omega. 
	\end{cases}
	\end{equation}
Furthermore, if
\begin{equation}\label{eq:pro51 ass}
 k = \Oh(\varepsilon^{\zeta_1} ) \quad\mbox{and} \quad |\lambda_1|= \Oh(\varepsilon^{\zeta_2} ) ,\quad \zeta_j\in \mathbb R_+,\quad j=1,2, 
\end{equation}
then it holds that
 \begin{equation}\label{eq:v est}
 	\|v\|_{H^1( \Omega )} = o(\varepsilon ), 
 \end{equation}
 where $\lim\limits_{\varepsilon \rightarrow 0}\frac{o(\varepsilon )}{\varepsilon }=0.$
	\end{proposition}
	
	\begin{proof}
	Since $\Delta \psi(\mathbf x; \varepsilon,k,\mathbf d )+ k^2 \psi(\mathbf x; \varepsilon,k,\mathbf d )=0 $ in $\Omega$, it  is ready to see that $v$ fulfills \eqref{eq:sec 5 v}. According to Proposition  \ref{pro:well-posedness}, one has 
	\begin{equation}\label{eq:sec 5 u est}
		\Vert u \Vert_{H^1(\Omega)} \leq  C \Vert \psi \Vert_{H^{\frac{1}{2}}(\Omega)} =\Oh(\varepsilon ).
	\end{equation}
By elliptic regularity of \eqref{eq:sec 5 v}, using \eqref{eq:sec 5 u est} and \eqref{eq:pro51 ass}, it yields that
	$$
	\|v\|_{H^1( \Omega )} \leq  k^2 \|\psi \|_{H^1( \Omega )}+\sum\limits_{j=1}^N |\lambda_j | \|u\|_{H^1( \Omega )}^j=o(\varepsilon )
	$$
	which completes the proof of this proposition. 
	\end{proof}

In the following we show that the admissible conditions introduced in previous sections can be fulfilled for the case that the nonlinear anomaly $f(\mathbf x, u)$ is characterized  by \eqref{eq:sec 5}. Under this situation, {Assumption A}  can be implied by {Assumption B} directly. Hence we first consider {Assumption B} in the proposition below.

\begin{proposition}
Suppose that $f(\mathbf x, u)$ in \eqref{eq:sec 5} takes the form \eqref{eq:form2} and the set $\{\varepsilon_1,\ldots, \varepsilon_N\}$ is pairwise different, where $\varepsilon_j\in \mathbb R_+$ with $\varepsilon_j  \ll 1$ and $N$ is an index related to $f(\mathbf x, u)$.  Let  $u_j$ be  the solution to \eqref{eq:sec 5} associated with $\psi(\mathbf x;\varepsilon_j,k,\mathbf d )= \varepsilon_j e^{\mathrm i k \mathbf x  \cdot \mathbf d} $, 	then {Assumption B} is fulfilled under the condition \eqref{eq:pro51 ass} and $k^2\neq  \lambda_1$. 
\end{proposition}
\begin{proof}
Under the assumption  \eqref{eq:pro51 ass}, for $\mathbf x_c \in \partial D$, from Proposition \ref{pro:51},  it  can be derive that
\begin{align}\label{eq:56 expan}
k^2 u_{j}(\mathbf x_c)-f(\mathbf x_c, u_{j}(\mathbf x_c))=(k^2 -\lambda_1) \psi(\mathbf x_c;\varepsilon_j,k,\mathbf d )+k^2 v_j + R, 
\end{align}
where $u_j=\psi(\mathbf x_c;\varepsilon_j,k,\mathbf d )+v_j$ and $R= -\sum_{\ell=2}^N \lambda_j (\psi(\mathbf x_c;\varepsilon_j,k,\mathbf d )+   v )^\ell
$. Here $v_j$ satisfies \eqref{eq:sec 5 v} and \eqref{eq:v est}. Therefore, from \eqref{eq:56 expan}, one can readily know that
\begin{align}\notag
	k^2 u_{j}(\mathbf x_c)-f(\mathbf x_c, u_{j}(\mathbf x_c))=(k^2 -\lambda_1) \psi(\mathbf x_c;\varepsilon_j,k,\mathbf d )+o(\varepsilon )\neq 0
\end{align}
by  noting $k^2\neq  \lambda_1$ and $\psi(\mathbf x_c;\varepsilon_j,k,\mathbf d ) \neq 0$.

Since  $\varepsilon_i\neq \varepsilon_j$ for any $i\neq j$  satisfying $i,j\in \{1,\ldots,N\}$, from Proposition \ref{pro:51},  we can obtain that 
\begin{equation}\notag
	\prod_{1\leq i\leq j\leq N} \left(u_{j} (\mathbf{x}_c)-u_{i}(\mathbf{x}_c)\right)=	\prod_{1\leq i\leq j\leq N} \left[(\varepsilon_j-\varepsilon_i ) e^{\mathrm i k \mathbf x_c \cdot \mathbf d}  +v_{j} (\mathbf{x}_c)-v_{i}(\mathbf{x}_c)\right]\neq 0
\end{equation}
by virtue of $\|v_i\|_{H^1( \Omega )} = o(\varepsilon_i )$ and $\|v_j\|_{H^1( \Omega )} = o(\varepsilon_j)$.

Therefore, the two admissible conditions  in {Assumption B} are fulfilled. 
\end{proof}

We can  adopt a similar argument for  proving Proposition \ref{pro:51} to validate the following proposition.

\begin{proposition}\label{pro:53 re}
 Consider the semilinear elliptic equation \eqref{eq:sec 5}, where the anomalous inclusion $D$ possesses a polygonal-nest or polyhedral-nest structure of the class $\mathcal A$ described by Definition \ref{def:3} satisfying
 \begin{equation}\label{eq:D nest}
 \Sigma_N\Subset\Sigma_{N-1}\Subset\cdots\Subset\Sigma_2\Subset\Sigma_1=D. 
 \end{equation}
 and $f(\mathbf x, u)$ is characterized by \eqref{eq:f sec4}. Here $ \Sigma_\ell$ is a polygon in $\mathbb R^2$ or  a polyhedron in $\mathbb R^3$, $\ell\in \{1,\ldots, N\}$. 
	Suppose that $u \in H^{1}(\Omega )$ is the solution to   \eqref{eq:sec 5}, which  satisfies $u=\psi(\mathbf x; \varepsilon,k,\mathbf d ) +v$. Then $v \in H^1(\Omega )$ fulfills
	\begin{equation}\label{eq:sec 5 v C}
	\begin{cases}
	\Delta v + k^2 v \chi_{\Omega \backslash D } = k^2 \psi \chi_D - \sum\limits_{\ell=1}^N \sum\limits_{j=1}^{M_\ell} \lambda_j^{(\ell)} u^j \chi_{U_\ell} - \sum\limits_{j=1}^{M_N} \lambda_j^{(N)} u^j \chi_{\Sigma_N}, & \quad \mbox{ in }\quad \Omega,\\
	v = 0 & \quad \mbox{ on }\quad \partial\Omega,
	\end{cases}
	\end{equation}
	where $U_\ell:=\Sigma_\ell\backslash\overline{\Sigma_{\ell+1}}$, $M_\ell \in \mathbb N$, and $1\leq \ell\leq N-1.$
Furthermore, if
\begin{equation}\label{eq:pro53 ass}
	 k =\Oh(\varepsilon^{\zeta_0} ) \quad\mbox{and} \quad |\lambda_1^{(\ell)}|=\Oh(\varepsilon^{\zeta_\ell} ) , \quad \ell\in \{1,\ldots,N\}, 
\end{equation}
where $\zeta_j\in \mathbb R_+$ for $j\in \{0,1,\ldots,N\}$,  
then it yields that
 \begin{equation}\label{eq:v est 53}
 	\|v\|_{H^1( \Omega )} = o(\varepsilon ).  
 \end{equation}
 
 Similarly, consider the semilinear elliptic equation \eqref{eq:sec 5}, where the anomalous inclusion $D$ possesses a polygonal-nest or polyhedral-nest structure of the class $\mathcal B$ described by Definition \ref{def:4} satisfying
 $$
 \Sigma_N\Subset\Sigma_{N-1}\Subset\cdots\Subset\Sigma_2\Subset\Sigma_1=D. 
 $$
 and $f(\mathbf x, u)$ is characterized by \eqref{eq:f sec5}. 
	Suppose that $u \in H^{1}(\Omega )$ is the solution to   \eqref{eq:sec 5}, which  satisfies $u=\psi(\mathbf x; \varepsilon,k,\mathbf d ) +v$. Then $v \in H^1(\Omega )$ fulfills
	\begin{equation}\label{eq:sec 5 v C}
	\begin{cases}
	\Delta v + k^2 v \chi_{\Omega \backslash D } = k^2 \psi \chi_D - \sum\limits_{\ell=1}^{N-1}\lambda_\ell u\chi_{U_\ell}-f_N(\mathbf x, u) \chi_{\Sigma_N } , & \quad \mbox{ in }\quad \Omega,\\
	v = 0 & \quad \mbox{ on }\quad \partial\Omega,
	\end{cases}
	\end{equation}
	where $U_\ell:=\Sigma_\ell\backslash\overline{\Sigma_{\ell+1}}$, $1\leq \ell\leq N-1$, and
	$$
	f_N(\mathbf x, u)=\sum_{j=1}^N \lambda_j^{(N)} u^j,\quad \lambda_j^{(N)}\in\mathbb{C}. 
	$$
Furthermore, if
\begin{equation}\label{eq:pro53 ass D}
	 k= \Oh(\varepsilon^{\zeta_0} ) ,  \quad |\lambda_\ell|=\Oh(\varepsilon^{\zeta_\ell } ) , \quad \ell\in \{1,\ldots,N-1\}, \quad\mbox{and} \quad  \left|\lambda_1^{(N)} \right|=\Oh(\varepsilon^{\zeta_N} ) , 
\end{equation}
where $\zeta_j\in \mathbb R_+$ for $j\in \{0,1,\ldots,N\}$, 
then it yields that
 \begin{equation}\notag 
	\|v\|_{H^1( \Omega )} = o(\varepsilon ).  
 \end{equation}

	\end{proposition}

{Assumption C} can be satisfied under generic conditions introduced in the following proposition. 
	
\begin{proposition}\label{pro:54}
Assume that an anomalous inclusion $D \Subset \Omega $ possesses a polygonal-nest or polyhedral-nest structure of the class $\mathcal A$ described by Definition \ref{def:3}, namely \eqref{eq:D nest} holds. Suppose that $f(\mathbf x, u)$ in \eqref{eq:sec 5} is characterized by   \eqref{eq:f sec4}  and the set $\mathcal E:=\cup_{\ell=1}^N\left \{\varepsilon_1^{(\ell)},\ldots, \varepsilon_{M_\ell}^{(\ell)}\right \}$ is pairwise different, where $\varepsilon_j^{(\ell)}\in \mathbb R_+$ with $\varepsilon_j^{(\ell)}  \ll 1$.  Let  $u_{\psi_j^{(\ell  )}}$ be  the solution to \eqref{eq:sec 5} associated with $\psi_j^{(\ell)}:=\psi(\mathbf x;\varepsilon_j^{(\ell)},k,\mathbf d )= \varepsilon_j^{(\ell)} e^{\mathrm i k \mathbf x  \cdot \mathbf d} $, 	then {Assumption C} is fulfilled under the condition \eqref{eq:pro53 ass} and 
\begin{equation}\label{eq:515 ass}
	\mbox{ $k^2\neq  \lambda_1^{(1)}$ and $\lambda_1^{(\ell-1)}\neq \lambda_1^{(\ell)}$ with $\ell\in \{2,\ldots,N\}$.   }
\end{equation}
\end{proposition}
\begin{proof}
	Let $v_{\psi_j^{(\ell  )}}:=u_{\psi_j^{(\ell  )}}- \psi_j^{(\ell)}  $. According to Proposition \ref{pro:53 re},  $v_{\psi_j^{(\ell  )}}$ fulfills 
	\begin{equation}\notag 
	\begin{cases}
	\Delta v_{\psi_j^{(\ell  )}} + k^2 v_{\psi_j^{(\ell  )}}  \chi_{\Omega \backslash D } = k^2 \psi_j^{(\ell)}  \chi_D - \sum\limits_{\ell=1}^N \sum\limits_{m=1}^{M_\ell} \lambda_m^{(\ell)} u^m_{\psi_j^{(\ell  )}} \chi_{U_\ell}- \sum\limits_{m=1}^{M_N} \lambda_m^{(N)} u^m_{\psi_j^{(\ell  )}}  \chi_{\Sigma_N} , & \quad \mbox{ in }\quad \Omega,\\
	v_{\psi_j^{(\ell  )}} = 0 & \quad \mbox{ on }\quad \partial\Omega,
	\end{cases}
	\end{equation}
	where $U_\ell:=\Sigma_\ell\backslash\overline{\Sigma_{\ell+1}}$ and $1\leq \ell\leq N-1.$ Under \eqref{eq:pro53 ass} one has
	\begin{equation}\label{eq:v est 53 psij}
 	\left\|v_{\psi_j^{(\ell  )}}\right\|_{H^1( \Omega )}=o(\varepsilon_j^{(\ell  )} ).  
 \end{equation}
 Recall  that $\mathcal{V}( \partial \Sigma_\ell)$ is the vertex set of $\Sigma_\ell$, $\ell=1,\ldots,N$. For the three conditions in {Assumption C}, it can be directly obtained that
\begin{align*}
&k^2 u_{\psi_j^{(1)}}(\mathbf y_c)-f(\mathbf y_c, u_{\psi_j^{(1)}}(\mathbf y_c)) \big|_{U_{1}} =\left (k^2 -\lambda_1^{(1)} \right)   \psi_{\psi_j^{(1)}} +\left(k^2-\lambda_1^{(1)} \right)    v_{\psi_j^{(1)}}(\mathbf y_c)\\
&-\sum_{m=2}^{M_\ell}\lambda_m^{(\ell-1)}  u^m_{\psi_j^{(1  )}}(\mathbf y_c)=\left (k^2 -\lambda_1^{(1)} \right)   \psi_{\psi_j^{(1)}}+o\left(\varepsilon_j^{(1)} \right) \neq 0,\ \ 1\leq j\leq M_1,\quad  \mathbf \forall \mathbf y_c\in \mathcal{V}( \partial \Sigma_1),  \\
&\sum_{m=1}^{M_\ell}\lambda_m^{(\ell-1)}  u^m_{\psi_j^{(\ell  )}}(\mathbf x_c)-f(\mathbf x_c, u_{\psi_j^{(\ell)}}(\mathbf x_c))\big|_{U_{\ell}}=  \left(\lambda_1^{(\ell-1)}  - \lambda_1^{(\ell)} \right ) \psi_{\psi_j^{(\ell  )}} (\mathbf x_c)+ \big (\lambda_1^{(\ell-1)}  - \lambda_1^{(\ell)} \big  ) \\
&\times v_{\psi_j^{(\ell  )}}(\mathbf x_c) +\sum_{m=2}^{M_\ell}\lambda_m^{(\ell-1)}  u^m_{\psi_j^{(\ell  )}}(\mathbf x_c) - \sum_{m=2}^{M_{\ell+1} }\lambda_m^{(\ell)}  u^m_{\psi_j^{(\ell  )}}(\mathbf x_c)=  \left(\lambda_1^{(\ell-1)}  - \lambda_1^{(\ell)} \right ) \psi_{\psi_j^{(\ell  )}} (\mathbf x_c)+  o\left(  \varepsilon_j^{(\ell)}  \right )  \\
& \neq 0,\  1\leq j\leq M_{\ell}, \ \forall  \mathbf{x}_c\in \mathcal{V}( \partial \Sigma_\ell),\  \ell=2,\ldots,N-1, \\
& \prod_{1\leq i\leq j\leq M_\ell } \left(u_{\psi_j^{(\ell)}}(\mathbf{x}_c)-u_{\psi_i^{(\ell)}}(\mathbf{x}_c)\right)=\prod_{1\leq i\leq j\leq N} \left[(\varepsilon_j^{(\ell)}-\varepsilon_i^{(\ell)} ) e^{\mathrm i k \mathbf x_c \cdot \mathbf d}  +v_{\psi_j^{(\ell  )}}(\mathbf x_c)-v_{\psi_i^{(\ell  )}}(\mathbf x_c)\right]\\
& \neq 0,\quad \forall  \mathbf{x}_c\in \mathcal{V}( \partial \Sigma_\ell),	\quad  \ell=1,\ldots,N, 
\end{align*}
by using \eqref{eq:515 ass}  and \eqref{eq:v est 53 psij}. 
\end{proof}

Using Proposition \ref{pro:53 re} and following a similar argument for Proposition \ref{pro:54}, we can show that  {Assumption D} can be satisfied under certain generic scenarios in the following proposition. The detailed proof of this proposition is omitted. 
\begin{proposition}
Assume that an anomalous inclusion $D \Subset \Omega $ possesses a polygonal-nest or polyhedral-nest structure of the class $\mathcal B$ described by Definition \ref{def:4}, namely \eqref{eq:D nest} holds. Suppose that $f(\mathbf x, u)$ in \eqref{eq:sec 5} is characterized by   \eqref{eq:f sec5}, where
\begin{align}\notag
f(\mathbf x ,u)\big|_{ U_\ell}&=\lambda_\ell  u,\quad \lambda_\ell \in\mathbb{C},\quad U_\ell:=\Sigma_\ell\backslash\overline{\Sigma_{\ell+1}}, \quad \lambda_\ell \neq \lambda_{\ell+1}, \quad 1\leq \ell\leq N-1, \\ 
	f(\mathbf x ,u)\big|_{ \Sigma_N}&=f_N(\mathbf x, u)=\sum_{j=1}^{M_N} \lambda_j^{(N)} u^j,\quad \lambda_j^{(N)} \in\mathbb{C},\quad M_N\in \mathbb N. \notag
\end{align}
  Let  $u_{\psi}$ be  the solution to \eqref{eq:sec 5} associated with $
\psi(\mathbf x;\varepsilon,k,\mathbf d )= \varepsilon e^{\mathrm i k \mathbf x  \cdot \mathbf d} $, 	then {Assumption D} is fulfilled under the conditions \eqref{eq:pro53 ass D} 
\begin{equation}\notag 
	\mbox{ $k^2\neq  \lambda_1^{(1)}$ and $\lambda_{N-1}\neq \lambda_1^{(N)}$.   }
\end{equation}
\end{proposition}

\section*{Acknowledgment}
The work of H. Diao is supported by a startup fund from Jilin University and NSFC/RGC Joint Research Grant No. 12161160314.   The work of H. Liu is supported by the Hong Kong RGC General Research Funds (projects 12302919, 12301420 and 11300821),  the NSFC/RGC Joint Research Fund (project N\_CityU101/21), the France-Hong Kong ANR/RGC Joint Research Grant, A-HKBU203/19.

\section*{Appendix}
In this section we prove the well-posedness of the forward  semilinear elliptic boundary value problem  with small boundary data, which is introduced in \eqref{eq:helm1}. Let $\mathrm{Q}$ be the semilinear elliptic operator given by 
\begin{equation}\label{eq:Q}
\mathrm{Q}(u):= \Delta u + a(\mathbf{x}, u),
\end{equation}
where $a(\mathbf x, u)$ is $ C^1$-continuous with respect to $u$ for a fixed $\mathbf x\in \Omega$ and $\partial_u a(\mathbf x, u) \in L^\infty (\Omega ).$ Moreover, we assume that the nonlinear term $a$ satisfies the following two conditions:
\begin{equation}\label{eq:a}
a(\mathbf{x} , 0) = 0,
\end{equation}
\begin{equation}\label{eq:map}
\mbox{ the } \mbox{ map }v \mapsto \Delta v + \partial_u a(\mbox{ . } , u) v \mbox{  is  injective on }  H_0^1({\Omega}).
\end{equation}
Indeed, from \eqref{eq:a}, one can directly know that  $u\equiv0$ is a solution of \eqref{eq:helm1} when the Dirichlet data is zero.  The condition \eqref{eq:map} guarantees that the linearized equation of \eqref{eq:helm1}  at $u\equiv0$ is well-posed. The next result considers mappings between Banach spaces which are Fr\'{e}chet differentiable. We refer the reader to \cite[Section 1.1]{Hor} and \cite[Section 10] {Renrog} for basics about Fr\'{e}chet differentiability.

\begin{proposition}(Well-posedness  of the semilinear elliptic boundary value problem   \eqref{eq:helm1} with small boundary data)\label{pro:well-posedness}
	Let $\Omega \Subset \mathbb{R}^n$, $n=2,3$ be a bounded Lipschitz domain and let $\mathrm{Q}$ be the semilinear elliptic operator given by \eqref{eq:Q} satisfying \eqref{eq:a} and \eqref{eq:map}. There exist constants  $\delta, C>0$ such that for any $\psi$ in the set
	\begin{equation}\nonumber
	U_{\delta}:= \{ h \in H^{\frac{1}{2}} (\partial\Omega); \Vert h\Vert_{H^{\frac{1}{2}} (\partial\Omega)} < \delta\},
	\end{equation}
	there is a solution $u=u_{\psi}$ of
	\begin{equation}\label{eq:ua}
	\begin{cases}
	\Delta u + a(\mathbf{x} , u) = 0 & \quad \mbox{ in }\quad \Omega,\\
	u = \psi & \quad \mbox{ on }\quad \partial\Omega,
	\end{cases}
	\end{equation}
	which satisfies
	\begin{equation}\nonumber
	\Vert u \Vert_{H^1(\Omega)} \leq  C \Vert \psi \Vert_{H^{\frac{1}{2}}(\Omega)}.
	\end{equation}
	The solution $u_{\psi}$ is unique within the class $\{w \in H^1(\Omega); \Vert w \Vert_{H^1(\Omega)} \leq C \delta\}$.
\end{proposition}

\begin{proof}
	We adopt the  implicit function theorem in Banach spaces \cite[Theorem 10.6]{Renrog} to prove the existence. 
Introduce the following map
	\begin{equation}\nonumber
	F~:~ H^{\frac{1}{2}} (\partial\Omega) \times  H^1(\Omega) \rightarrow H^{-1}(\Omega) \times H^{\frac{1}{2}} (\partial\Omega),   \quad F({\psi,u}) = (\mathrm{Q}(u) , u\big |_{\partial\Omega} - \psi).
	\end{equation}
	We  first show that the image of $F$ belongs to  $Z$. Recall that  $a(\mathbf{x},u)$ defined in \eqref{eq:aaa1} is $C^1$-continuous with respect to $u$ for a fixed $\mathbf x\in \Omega$.  One know that 
	\begin{equation}\nonumber
	u \rightarrow a(\mathbf{x} , u)
	\end{equation}
maps $H^1(\Omega)$ to $L^2(\Omega)$. Since $u \in H^1(\Omega)$, $\mathrm{Q}(u) \in H^{-1}(\Omega)$ is defined in \eqref{eq:Q} and $(u|_{\partial\Omega} - \psi) \in H^{\frac{1}{2}}(\partial\Omega)$. Therefore $F$ is well defined.
	
	Next, we prove that $F$ is continuously differentiable. Recall that  $a(\mathbf x, u)$ is $ C^1$-continuous with respect to $u$ for a fixed $\mathbf x\in \Omega$, which implies that
	$$
	a(\mathbf x, u+v)=a(\mathbf x, u)+\partial_u a(\mathbf x, u)v+\|v\|_{H^1(\Omega ) }\delta(v),\quad\lim_{\|v\|_{H^1(\Omega ) }\rightarrow 0 } \delta (v)=0.
	$$
	Therefore, it yields that
	\begin{align}\label{eq:F der}
		F(\psi+\varphi, u+ v )=F(\psi, u)+(\Delta v+\partial_u a(\mathbf x, u)v,  v\big| _{\partial \Omega }-\varphi )+(\|v\|_{H^1(\Omega ) }\delta(v), 0).
	\end{align}	
	Noting that $a(\mathbf x, u)$ is $ C^1$-continuous with respect to $u$ for a fixed $\mathbf x\in \Omega$, we  can conclude that  $F$ is continuously differentiable from $H^{1/2}(\partial \Omega ) \times H^1(\Omega )$ to $H^{-1} (\Omega ) \times H^{1/2}(\partial \Omega ) $.

	From \eqref{eq:F der}, the linearization of $F$ at $(0,0)$ is
	\begin{equation}\nonumber
	D_u F|_{(0,0)}(v) = ( \Delta v + \partial_u a(\mathbf{x}, 0)v, v|_{\partial\Omega}).
	\end{equation}

	In the  following we show that $D_u F|_{(0,0)}$ is a homeomorphism from $H^1( \Omega	) $ to  $H^{-1} (\Omega ) \times H^{1/2}(\partial \Omega ) $ under the condition \eqref{eq:map}. To this end, consider the following Dirichlet boundary value problem 
	\begin{equation}\label{eq:v  eq}
	\begin{cases}
	\Delta v +\partial_u a(\mathbf{x} , 0)v = f & \quad \mbox{ in } \Omega,\\
	v = \varphi & \quad \mbox{ on } \partial\Omega,
	\end{cases}
	\end{equation}
where $f\in H^{-1}(\Omega)$ and $\varphi\in H^{1/2}(\partial  \Omega ) $. Suppose that there exists a solution to \eqref{eq:v eq}, then the solution is unique by using \eqref{eq:map}. Therefore, utilizing Fredholm alternative (cf. \cite[Proposition 1.9]{Tay}), one can show that  there exist a solution to  \eqref{eq:v  eq} in $H^1(\Omega)$ for any source in $H^{-1}(\Omega)$ and the Dirichlet data in $H^{1/2}(\partial \Omega )$.
	
According to	the implicit function theorem in Banach spaces \cite[Theorem 10.6]{Renrog}, we  know that there is a $\varepsilon >0$ and an open ball $B_\varepsilon = B(0,\varepsilon )\subset H^{1/2} (\partial \Omega ) $ and a continuously differential mapping $ \mathcal T: B_\varepsilon \rightarrow  H^{1}(\Omega )$ such that 
	\begin{equation}\nonumber
	F(\psi, \mathcal T(\psi))= (0,0)
	\end{equation}
	under the  condition $\Vert \psi\Vert_{H^{1/2}(\partial\Omega)} \leq \delta$. 
	Since $\mathcal T$ is Lipschitz continuous and $\mathcal T(0) = 0$, $u= \mathcal T(\psi)$ satisfies
	\begin{equation}\nonumber
	\Vert u \Vert_{H^1(\Omega)} \leq C \Vert \psi \Vert_{H^{1/2}(\partial \Omega)}.
	\end{equation}
	Furthermore, one  can  claim that $u =\mathcal T(\psi)$ is the unique  solution to $ F(\psi, S(\psi))=(0,0)$ under the assumption $\Vert \psi \Vert_{H^\frac{1}{2}(\partial \Omega)} \leq \delta$   by necessarily refining the parameter $\delta$, where $ \Vert u \Vert_{H^1(\Omega)} \leq C \delta$.
	
	The proof is complete. 
	\end{proof}

\end{document}